\documentclass[a4paper,12pt,reqno]{amsart}

\usepackage[T1]{fontenc}
\usepackage[utf8]{inputenc}
\usepackage{lmodern}
\usepackage[english]{babel}

\usepackage[%
	left=2.15cm,       
	right=2.15cm,      
	top=3.5cm,        
	bottom=3.5cm,     
	heightrounded,    
	bindingoffset=0mm 
]{geometry}

\usepackage{amssymb}
\usepackage{mathtools}
\usepackage{mathrsfs}
\usepackage{xfrac}

\usepackage{dutchcal}
\newcommand{\nupla}{\mathcal}

\usepackage{textgreek}

\usepackage[normalem]{ulem} 

\usepackage{microtype}

\usepackage[
colorlinks=true,
urlcolor=teal,
citecolor=magenta,
linkcolor=teal,
linktocpage,
pdfpagelabels,
bookmarksnumbered,
bookmarksopen,
breaklinks=true]
{hyperref}

\usepackage[capitalize,nameinlink,noabbrev]
{cleveref}

\usepackage{bookmark}

\usepackage[initials,msc-links,
]{amsrefs}

\usepackage[dvipsnames]{xcolor}
\usepackage{graphicx}
\usepackage{enumitem}

\numberwithin{equation}{section}

\theoremstyle{plain}
	\newtheorem{theorem}{Theorem}[section]
	\newtheorem{lemma}[theorem]{Lemma}	\newtheorem{proposition}[theorem]{Proposition}
\newtheorem{corollary}[theorem]{Corollary}

\theoremstyle{definition}
\newtheorem{definition}[theorem]{Definition}

\newtheorem{remark}[theorem]{Remark}

\newcommand{\lstort}{{\mathscr{L}}}
\newcommand{\de}{\,\mathrm{d}}
\newcommand{\E}{{\mathcal E}}

\newcommand{\eps}{\varepsilon}
\newcommand{\loc}{{\mathrm{loc}}}
\newcommand{\N}{\mathbb{N}}
\newcommand{\ort}{\R^N_{+}}
\newcommand{\R}{\mathbb{R}}

\DeclareMathOperator{\argmin}{arg\, min}

\DeclareMathOperator{\per}{{\mathrm{Per}}}

\DeclarePairedDelimiter{\scalar}{\langle}{\rangle}
\DeclarePairedDelimiter{\set}{\{}{\}}

\usepackage[final]{showlabels}

\showlabels{bib}

\allowdisplaybreaks

\usepackage{comment}

\begin{document}

\title{On the $N$-Cheeger problem for component-wise increasing norms}

\author[G.~Saracco]{Giorgio Saracco}
\address[G.~Saracco]{Dipartimento di Matematica e Informatica ``Ulisse Dini'' (DIMAI), Università degli Studi di Firenze, Viale Morgagni 67/A, 50134 Firenze (FI), Italy}
\email{giorgio.saracco@unifi.it}

\author[G.~Stefani]{Giorgio Stefani}
\address[G.~Stefani]{Scuola Internazionale Superiore di Studi Avanzati (SISSA), via Bonomea 265, 34136 Trieste (TS), Italy}
\email{gstefani@sissa.it {\normalfont or} giorgio.stefani.math@gmail.com}

\date{\today}



\thanks{\textit{Acknowledgments}.
The authors are members of the Istituto Nazionale di Alta Matematica (INdAM), Gruppo Nazionale per l'Analisi Matematica, la Probabilit\`a e le loro Applicazioni (GNAMPA).
The authors have received funding from INdAM under the INdAM--GNAMPA Project 2024 \textit{Ottimizzazione e disuguaglianze funzionali per problemi geometrico-spettrali locali e non-locali}, codice CUP\_E53\-C23\-001\-670\-001.
The first-named author has received funding from INdAM under the INdAM--GNAMPA Project 2023 \textit{Esistenza e propriet\`a fini di forme ottime}, codice CUP\_E53\-C22\-001\-930\-001, and from Universit\`a di Trento (UNITN) under the Starting Grant Giovani Ricercatori 2021 project \textit{WeiCAp}, codice CUP\_E65\-F21\-00\-41\-60\-001.
The second-named author has received funding from INdAM under the INdAM--GNAMPA 2023 Project \textit{Problemi variazionali per funzionali e operatori non-locali}, codice CUP\_E53\-C22\-001\-930\-001, and from the European Research Council (ERC) under the European Union’s Horizon 2020 research and innovation program (grant agreement No.~945655).
The present research started during a visit of the first-named author at the Scuola Internazionale Superiore di Studi Avanzati (SISSA). The first-named author wishes to thank the institution for its kind hospitality.}

\begin{abstract}
We study Cheeger and $p$-eigenvalue partition problems depending on a given evaluation function $\Phi$ for $p\in[1,\infty)$.
We prove existence and regularity of minima, relations between the problems, convergence, and stability with respect to $p$ and to $\Phi$.
\end{abstract}

\hspace*{-2em}
{
\begin{minipage}{0.71\textwidth}
\vspace{-20ex}
{\tiny This is a preprint of the article available at DOI \href{https://doi.org/10.1016/j.matpur.2024.06.008}{10.1016/j.matpur.2024.06.008}. Due to a miscommunication, we received one report \emph{after} publication. This version incorporates the report’s feedback and corrects the statement of \cref{res:choose_t_i} by adding a missing hypothesis.\par}
\end{minipage} 
}

\maketitle

\section{Introduction}

\subsection{Cheeger problem}

The \emph{Cheeger constant} of a measurable set $\Omega\subset\R^d$ is defined as
\begin{equation}
\label{eqi:cheeger}    
h(\Omega) = \inf\set*{\frac{\per(E)}{|E|}: E\subset \Omega,\ |E|>0},
\end{equation}
where $\per(E)$ and $|E|$ are the distributional perimeter (refer to~\cite{Mag12book}) and  the $d$-di\-men\-sio\-nal Lebesgue measure of $E$, respectively.
The study of~\eqref{eqi:cheeger} has drawn a lot of attention in the past decades, see~\cites{Leo15,Par11,FPSS24} for an exhaustive presentation.

Even though~\eqref{eqi:cheeger} is purely geometrical in the stated form, $h(\Omega)$ has remarkable spectral properties. 
As noticed for $p=2$ by Maz'ya~\cites{Gri89,Maz62-1,Maz62-2} in $\R^d$ and by Cheeger~\cite{Che70} on a Riemannian $d$-manifold, and later extended to any $p\in(1,\infty)$ in~\cites{Avi97, LW97},
\[
\left(\frac{h(\Omega)}{p}\right)^p \le \lambda_{1,p}(\Omega),
\]
where 
\begin{equation}
\label{eqi:eigen}
\lambda_{1,p}(\Omega)    
=
\inf\set*{\|\nabla u\|_{L^p} : u\in W^{1,p}_0(\Omega),\ \|u\|_{L^p}=1}
\end{equation}
is the \emph{first eigenvalue of the Dirichlet $p$-Laplacian} on~$\Omega$. 
Actually, as proved in~\cite{KF03},
\begin{equation}
\label{eqi:limit}
h(\Omega) = \lim_{p\to 1^+} \lambda_{1,p}(\Omega),
\end{equation}
so that $h(\Omega)$ may be thought of as the \emph{first eigenvalue of the Dirichlet $1$-Laplacian} on~$\Omega$.

\subsection{Partition problems}

A natural extension of~\eqref{eqi:cheeger} consists in finding \emph{clusters} of $\Omega$ that minimize a combination of their isoperimetric ratios, see~\cites{FPS23a,FPS23b,NPST22,NPST23}.

Given $N\in\N$, an \emph{$N$-set} of $\Omega$ is an $N$-tuple $\nupla E = (\nupla E_1,\dots,\nupla E_N)$ of pairwise disjoint subsets with positive measure $\nupla E_i\subset\Omega$, called \emph{chambers} of~$E$.
If, in addition, $\per(\nupla E_i)<\infty$ for each $i=1,\dots,N$, then $\nupla E$ is an \emph{$N$-cluster} of $\Omega$.
Given a \emph{reference function} $\Phi\colon\ort \to [0,\infty)$ (for instance, any $q$-norm in $\R^N$), we consider
\begin{equation}
\label{eqi:H}
\tag{$H$}
H^{\Phi, N}(\Omega) = \inf\set*{\Phi\left(\frac{\per(\nupla E_1)}{|\nupla E_1|},\dots,\frac{\per(\nupla E_N)}{|\nupla E_N|}\right) : \nupla E\ \text{is an $N$-cluster of}\ \Omega}.
\end{equation}

Concerning spectral analogs of~\eqref{eqi:H}, we have two possible natural formulations.
On the geometric side, for any $p\in[1,\infty)$, we introduce
\begin{equation}
\label{eqi:Lstort}
\tag{$\lstort_p$}
\lstort^{\Phi,N}_{1,p}(\Omega)
=
\inf\set*{
\Phi\left(\lambda_{1,p}(\nupla E_1),\dots, \lambda_{1,p}(\nupla E_N)\right) : \nupla E\ \text{is an $N$-set of}\ \Omega},
\end{equation}
while, on the functional side, having in mind~\eqref{eqi:eigen}, we consider 
\begin{equation}
\label{eqi:L}
\tag{$\Lambda_p$}
\Lambda^{\Phi,N}_{1,p}(\Omega)
=
\inf\set*{
\Phi\left(\|\nabla \nupla u_1\|_p^p, \dots, \|\nabla \nupla u_N\|_p^p\right) : \nupla u = (\nupla u_1, \dots, \nupla u_N)},
\end{equation}
where the infimum runs on $N$-tuples $\nupla u = (\nupla u_1, \dots, \nupla u_N)$ of functions in $W^{1,p}_0(\Omega)$ (or $BV_0(\Omega)$, for $p=1$) functions with pairwise disjoint supports and unitary $p$-norm.

\subsection{Previous results and main aim}

We are interested in studying~\eqref{eqi:H}, \eqref{eqi:Lstort} and~\eqref{eqi:L} and their relations under minimal assumptions on the reference function $\Phi$. 

Some partial results are already available in the literature. 
For the supremum norm, i.e., $\Phi= \|\cdot\|_\infty$, similarly to~\eqref{eqi:limit}, the convergence of~\eqref{eqi:Lstort} to~\eqref{eqi:H} as $p\to1^+$ is established in~\cite{BP18}, while several properties of the constant $H^{\Phi, N}(\Omega)$ are studied in~\cite{P10} (see also the recent work~\cite{FR23}).
For the $1$-norm $\Phi=\|\cdot\|_1$, the equivalence between~\eqref{eqi:H} and~\eqref{eqi:L} for $p=1$, as well as the relation between the superlevel sets of minimizers of~\eqref{eqi:L} with clusters minimizing~\eqref{eqi:H}, are proved in~\cite{CL19}, while the behavior of $H^{\Phi,N}(\Omega)$ as $N\to\infty$ is studied in~\cite{BF19}.

In passing, we mention that similar problems are considered in~\cite{BBF20}, where the Cheeger constant is replaced by the \emph{$\alpha$-Cheeger constant} for some $\alpha>0$, that is, the infimum on ratios of perimeter over the $\alpha$-th power of the volume (see~\cite{PS17} for an account).

Our main aim is to extend the results of~\cites{BP18,CL19} to general reference functions, that may not even be norms. Moreover, we prove the stability of the constants as the reference function $\Phi$ changes, bridging the gap between the available results.

Even though we work in the Euclidean space, most of the results can be extended within the abstract framework of~\cite{FPSS24}. For the sake of completeness, at the end of every section, we remark how our results can be extended to more general settings.

\subsection{Organization of the paper}

In \cref{sec:notation}, we set the notation and the basic definitions. 
In particular, we list the assumptions on the reference function $\Phi$  we will use throughout the paper (see \cref{ssec:reference_phi}). 
In \cref{sec:p=1}, we study the equivalence between~\eqref{eqi:H} and~\eqref{eqi:Lstort} for $p=1$, also providing regularity properties for their minimizers.   
In \cref{sec:caroccia-littig}, we study~\eqref{eqi:L} for $p=1$, extending the results of~\cite{CL19} to a general reference function $\Phi$ and proving boundedness of minimizers.
In \cref{sec:bobkov-parini}, we prove the equivalence between~\eqref{eqi:Lstort} and~\eqref{eqi:L} for $p>1$ and the boundedness of minimizers of~\eqref{eqi:L}.
We also generalize the convergence result of~\cite{BP18} as $p\to1^+$, both of the constants and their minimizers.
Lastly, in \cref{sec:stability}, we tackle the stability of~\eqref{eqi:H}, \eqref{eqi:Lstort} and~\eqref{eqi:L} with respect to a varying family of reference functions, proving convergence of the constants and their minimizers under natural equicoercivity assumptions.

\section{Notation and definitions}
\label{sec:notation}

\subsection{Notation}
Given $d\in \N$, we let $|E|$ and $\per(E)$ be the $d$-dimensional Lebesgue measure and the Euclidean perimeter of a Lebesgue measurable set $E\subset\R^d$, respectively. For the theory of sets of finite perimeter, we refer the reader to~\cite{Mag12book}.

We stress that, throughout the paper, we consider Lebesgue measurable sets only, and set inclusions are always meant in the measure-theoretic sense, i.e., $E\subset F$ if $|F\setminus E|=0$. Moreover, we shall always let $\Omega\subset\R^d$ be a fixed non-empty, bounded, and open set.

Given $N\in\N$ and $p\in[1,\infty]$, we let $\|\cdot\|_p\colon\R^N\to[0,\infty)$ be the usual $p$-norm on~$\R^N$, that is, for any $\nupla v\in\R^N$,
\begin{equation}
\label{eq:p_norm}
\|\nupla v\|_p
=
\begin{cases}
\left(\sum_{i=1}^N|\nupla v_i|^p\right)^{\frac1p},
&
\quad
p\in[1,\infty),
\\[1ex]
\max\set*{|\nupla v_i| : i=1,\dots, N},
&
\quad
p=\infty.
\end{cases}
\end{equation}
We consider the cone of $N$-vectors with non-negative components  
\begin{equation*}
\ort=\set*{\nupla v\in\R^N : \nupla v_i\ge0\ \text{for}\ i=1,\dots,N},
\end{equation*}
and we partially order its elements component-wise, that is, given $\nupla v,\nupla w\in\ort$,
\begin{align*}
\nupla v\le\nupla w, \qquad &\text{if $\nupla v_i\le\nupla w_i$ for all $i=1,\dots,N$,}
\\
\nupla v<\nupla w, \qquad &\text{if $\nupla v\le\nupla w$ and $\nupla v_i<\nupla w_i$ for some $i\in\set*{1,\dots,N}$.}
\end{align*} 
In particular, $\nupla v\ge 0$ for all $\nupla v\in\ort$.

\subsection{Cheeger constant}

We recall the following standard definition. 

\begin{definition}[Cheeger constant]
\label{def:h}
The \emph{Cheeger constant} of $F\subset\R^d$ is 
\begin{equation*}
h(F)= \inf\set*{
\frac{\per(E)}{|E|} : E\subset F,\ |E|>0}\in[0,\infty].
\end{equation*}
Any set $E\subset F$ with $|E|>0$ achieving the infimum is called a \emph{Cheeger set} of $F$.
\end{definition}

Note that $h(F)<\infty$ whenever $F\subset\R^d$ contains a viable competitor, i.e., $E\subset F$ with $\per(E)<\infty$ and $|E|>0$.
In particular, letting $\Omega$ be a non-empty, bounded, open set, one has $h(\Omega)<\infty$, since we may consider any ball contained in $\Omega$ as a viable competitor.

\subsection{\texorpdfstring{$N$}{N}-sets and \texorpdfstring{$N$}{N}-clusters}

Here we define the competitors considered in the paper.

\begin{definition}[$N$-set and $N$-cluster]
\label{def:cluster}
Given $F\subset\R^d$, an $N$-tuple
$\nupla E=(\nupla E_1, \dots, \nupla E_N)$
is an \emph{$N$-set} of $F$ if
$\nupla E_i\subset F$, $|\nupla E_i|>0$ and $|\nupla E_i\cap \nupla E_j|=0$ for $i\neq j$, and $i,j=1,\dots,N$. We shall call each $\nupla E_i$ a \emph{chamber} of the $N$-set $\nupla E$.

If additionally the perimeter of each chamber is finite, i.e., $\per(\nupla E_i)<\infty$ for $i=1,\dots,N$, we say that the $N$-set $\nupla E$ is an \emph{$N$-cluster} of $F$.
\end{definition}

Note that any set $F\subset\R^d$ with $|F|>0$ admits $N$-sets for any $N\in\N$. Furthermore, any non-empty, bounded, open set $\Omega$ admits $N$-clusters for any $N\in\N$, as one can consider the $N$-tuple $\nupla E$ given by $N$ disjoint balls contained in $\Omega$. 

Given any $N$-set $\nupla E$ of a set $F\subset\R^d$ as in \cref{def:cluster}, we let 
\begin{equation}\label{eq:def_M_i}
F^{\nupla E}_i= \bigcup_{j\neq i} \E_j
\quad
\text{for each}\
i=1,\dots,N.
\end{equation}
Note that $F_i^{\nupla E}\subset F$ and $\nupla E_i\subset F\setminus F_i^{\nupla E}$ for each $i=1,\dots,N$. The following definition was first introduced in~\cite{BP18}.

\begin{definition}[$1$-adjusted $N$-cluster]\label{def:1-adj}
An $N$-cluster $\nupla E$ of a set $F\subset\R^d$ is \emph{$1$-adjusted} if 
\begin{equation*}
\frac{\per(\nupla E_i)}{|\nupla E_i|} = h(F\setminus F_i^{\nupla E})
\quad
\text{for each}\
i=1,\dots,N.
\end{equation*}
\end{definition}

\begin{remark}\label{rem:1-adj_self}
If the $N$-cluster $\nupla E$ of a set $F\subset\R^d$ is $1$-adjusted as in \cref{def:1-adj}, then
\begin{equation}
\label{eq:1-adj_self}
h(\nupla E_i)
=
\frac{\per(\nupla E_i)}{|\nupla E_i|}
\quad
\text{for each}\
i=1,\dots,N.
\end{equation}
Indeed, since $\nupla E_i\subset F\setminus F^{\nupla E}_i$ and due to the monotonicity of the Cheeger constant with respect to set inclusions (see~\cite{Leo15}*{Prop.~3.5(i)}), we can write
\begin{equation}\label{eq:1_adj_h_chamber=complement}
\frac{\per(\nupla E_i)}{|\nupla E_i|} 
= 
h(F\setminus F_i^{\nupla E}) 
\le 
h(\E_i) 
\le 
\frac{\per(\nupla E_i)}{|\nupla E_i|},
\end{equation}
thus all inequalities are equalities.
In particular, each $\nupla E_i$ is a Cheeger set of $F\setminus F^{\nupla E}_i$ (and of itself, of course).
\end{remark}

\subsection{Reference function}
\label{ssec:reference_phi}

Throughout the paper, we let $\Phi\colon \ort \to [0, \infty)$ be the \emph{reference function}. From time to time, we will require $\Phi$ to possess some of the following properties:

\begin{enumerate}[label=(\textPhi.\arabic*)]

\item    
\label{p:lsc}
$\Phi$ is lower semicontinuous;

\item 
\label{p:c}
$\Phi$ is \emph{coercive}, i.e., there exists $\delta>0$ such that $\Phi(\nupla v)\ge\delta\|\nupla v\|_1$ for all $\nupla v\in\ort$;
    
\item
\label{p:m} 
$\Phi$ is \emph{increasing}, i.e., if $\nupla v,\nupla w\in\ort$ with $\nupla v\le\nupla w$, then $\Phi(\nupla v)\le\Phi(\nupla w)$.
 
\end{enumerate}

Properties \ref{p:lsc} and \ref{p:c} are quite natural to impose when dealing with existence results, as they guarantee lower semicontinuity and coercivity of the energy. Note that \ref{p:c}, once satisfied, holds with respect to any norm. Hence the choice of the $1$-norm in \ref{p:c} is made for convenience only. Property \ref{p:m} allows to compare different energies, and thus it is quite natural to impose when comparing different minimization problems.

In \cref{sec:bobkov-parini} (specifically, \cref{res:limit}) we need a stronger version of \ref{p:lsc}, while throughout \cref{ssec:skeleton_CL,ssec:boundedness_CL} a stronger one of \ref{p:m}. 
Precisely, we strengthen them as follows:

\begin{enumerate}[label=(\textPhi.1\textsuperscript{+})]

\item
\label{p:+}
$\Phi$ is continuous;
\end{enumerate} 

\begin{enumerate}[label=(\textPhi.3\textsuperscript{+})]

\item
\label{p:ms}
$\Phi$ is \emph{strictly increasing}, i.e.,  
if $\nupla v,\nupla w\in\ort$ with $\nupla v<\nupla w$, then $\Phi(\nupla v)<\Phi(\nupla w)$. 
\end{enumerate} 

Note that \ref{p:lsc} (actually, the stronger \ref{p:+}) and \ref{p:c} are met by any norm on $\R^N$.
However, not all norms on $\R^N$ satisfy~\ref{p:m}.
A counterexample for $N=2$ is given by
\begin{equation*}
(\nupla v_1,\nupla v_2)\mapsto\sqrt{4(\nupla v_1-\nupla v_2)^2+\nupla v_2^2},
\quad
\nupla v=(\nupla v_1,\nupla v_2)\in\R^2.
\end{equation*}
The $p$-norm~\eqref{eq:p_norm} satisfies~\ref{p:m} and, as long as $p<\infty$, also~\ref{p:ms}. On the other hand, it can be easily checked that there exist reference functions~$\Phi$ satisfying \ref{p:+}, \ref{p:c}, and \ref{p:ms} which are not norms on $\R^N$, since $1$-homogeneity is not necessarily needed.

We stress that every statement in the present paper contains the bare minimum hypotheses on the reference function for it to hold. Nevertheless, assuming \ref{p:+}, \ref{p:c}, and \ref{p:ms}, all results of the present paper hold true.

\subsection{\texorpdfstring{$(\Phi,N)$}{(Phi,N)}-Cheeger constant}

We generalize \cref{def:h} to $N$-clusters as follows. Loose\-ly speaking, given an $N$-cluster, one considers the $N$-dimensional vector given by the ratios of perimeter over measure of its chambers, and then evaluate it via a given reference function $\Phi$. 
\cref{def:H} below was considered for the first time for $\Phi=\|\cdot\|_p$ as in~\eqref{eq:p_norm} with $p=1$ in~\cites{C17,CL19} and with $p=\infty$ in~\cite{BP18}.

\begin{definition}[$(\Phi,N)$-Cheeger constant]
\label{def:H}
The \emph{$(\Phi,N)$-Cheeger constant} of a set $F\subset\R^d$ is 
\begin{equation}
\label{eq:H}
H^{\Phi,N}(F) 
= 
\inf\set*{
\Phi\left({\frac{\per(\nupla E)}{|\nupla E|}}\right) : \nupla E\ \text{is an $N$-cluster of $F$}
}\in[0,\infty],
\end{equation}
where, for brevity, we have set
\begin{equation*}
{\frac{\per(\nupla E)}{|\nupla E|}}
=
\left(\frac{\per(\nupla E_1)}{|\nupla E_1|},\dots,\frac{\per(\nupla E_N)}{|\nupla E_N|}\right).
\end{equation*}
A \emph{$\Phi$-Cheeger $N$-cluster} of $F$ is any $N$-cluster $\nupla E$ of $F$ achieving the infimum in~\eqref{eq:H}. 
\end{definition}

\begin{remark}
The notation and the definitions introduced in the present section can be restated \textit{verbatim} in the abstract setting of~\cite{FPSS24}.
\end{remark}

Given any set $F\subset \R^d$ admitting an $N$-cluster, we have $H^{\Phi,N}(F)<\infty$. In particular, this holds true for a non-empty, bounded, open set $\Omega$ as we can consider $N$ disjoint balls contained in $\Omega$.

\section{Existence, properties and regularity of minimizers}
\label{sec:p=1}

In this section, we study existence and regularity properties of minimizers. First, though, it is useful to observe that the $(\Phi,N)$-Cheeger constant has an alternate spectral-geometric definition. Indeed, we recall that $h(F)$ can be thought of as the first Dirichlet eigenvalue of the $1$-Laplacian (refer to~\cite{FPSS24}*{Sect.~5}). We anticipate that, in \cref{ssec:1_functional_eigenvalue}, we will give a further equivalent spectral-functional definition of the $(\Phi,N)$-Cheeger constant, where the competitors are given by suitable $N$-tuples of $BV$ functions.

\subsection{First \texorpdfstring{$1$}{1}-geometric \texorpdfstring{$(\Phi, N)$}{(Phi, N)}-eigenvalue}

Below we introduce the definition of the \emph{first $1$-geometric $(\Phi,N)$-ei\-gen\-va\-lue} of a set $F\subset\R^d$, and we shall see that, up to assuming \ref{p:m}, this is a viable alternative to \cref{def:H}. 
We remark that the following has been used as definition of $(\Phi,N)$-Cheeger constant in~\cites{P10,BP18} for the case $\Phi=\|\cdot\|_\infty$. 

\begin{definition}[First $1$-geometric $(\Phi,N)$-eigenvalue]
\label{def:lstort_1}
The \emph{first $1$-geometric $(\Phi,N)$-ei\-gen\-va\-lue} of a set $F\subset\R^d$ is 
\begin{equation}
\label{eq:lstort_1}
\lstort^{\Phi,N}_{1,1}(F)
=
\inf\set*{
\Phi({h(\nupla E)}) : \nupla E\ \text{is an $N$-set of $F$}
}\in[0,\infty],
\end{equation}
where, for brevity, we have set
\begin{equation*}
{h(\nupla E)}
=
\big(h(\nupla E_1),\dots,h(\nupla E_N)\big).
\end{equation*} 
A \emph{$(1,\Phi)$-eigen-$N$-set} of $F$ is any $N$-set $\nupla E$ of $F$ achieving the infimum in~\eqref{eq:lstort_1}.
A \emph{$(1,\Phi)$-eigen-$N$-cluster} of $F$ is any $(1,\Phi)$-eigen-$N$-set $\nupla E$ of $F$ which is also an $N$-cluster.
\end{definition}

Just as we did for the $(\Phi, N)$-Cheeger constant, we note that, if $F\subset \R^d$ is a set admitting an $N$-cluster, then $\lstort^{\Phi,N}_{1,1}(F)<\infty$. In particular, this holds true for a non-empty, bounded, open set $\Omega$.

We let the reader note that~\eqref{eq:lstort_1} is apparently different from~\eqref{eqi:Lstort} for $p=1$ in the introduction.
However, the two problems do coincide, as shown later on in \cref{res:lstort_1_N-set}.

It is worth noticing that, without loss of generality, one can consider $N$-clusters only in the above \cref{def:lstort_1} (provided that $|F|<\infty$) thanks to the following simple result.

\begin{proposition}\label{res:clusters_only}
Given a set $F\subset\R^d$ with $|F|<\infty$, for any $N$-set $\nupla E$ of $F$ with $h(\nupla E)\in\ort$, there exists an $N$-cluster $\tilde{\nupla E}$ of $F$ such that $h(\tilde{\nupla E})=h(\nupla E)$ and $\tilde{\nupla E}_i\subset\nupla E_i$ for $i=1,\dots,N$.
Consequently,
\begin{equation*}
\lstort^{\Phi,N}_{1,1}(F)
=
\inf\set*{
\Phi({h(\nupla E)}) : \nupla E\ \text{is an $N$-cluster of $F$}
}\in[0,\infty],
\end{equation*}
and it is thus not restrictive to work with $(1,\Phi)$-eigen-$N$-clusters of $F$ only.
\end{proposition}

\begin{proof}
If $\nupla E$ is an $N$-set of $F$ with $h(\nupla E)\in\ort$, then, by \cref{def:h,def:cluster}  each $\nupla E_i$ has positive measure and a subset with positive measure and finite perimeter. Moreover, the inclusion $\nupla E_i\subset F$ implies that $|\nupla E_i|<\infty$, and thus $\nupla E_i$ admits a Cheeger set $\tilde{\nupla E}_i\subset\nupla E_i$, see~\cite{FPSS24}*{Sect.~3.1},  so that $h(\nupla E_i)=h(\tilde{\nupla E}_i)$ for $i=1,\dots,N$. Hence $\tilde{\nupla E}=(\tilde{\nupla E}_1,\dots,\tilde{\nupla E}_N)$ is an $N$-cluster of $F$ such that $h(\nupla E)=h(\tilde{\nupla E})$.
We thus get that $\Phi(h(\tilde{\nupla E}))=\lstort^{\Phi,N}_{1,1}(F)$ and the conclusion follows.
\end{proof}

The following result proves that, assuming \ref{p:m}, \cref{def:H,def:lstort_1} are in fact equivalent on a non-empty, bounded, and open set $\Omega$, generalizing~\cite{P10}*{Prop.~3.5}. Moreover, a first relation between minimizers of the two problems is established.

\begin{proposition}[$H^{\Phi,N}(\Omega)=\lstort^{\Phi,N}_{1,1}(\Omega)$]
\label{res:H=lstort}
The following holds
\begin{equation*}
\lstort^{\Phi,N}_{1,1}(\Omega)\ge H^{\Phi,N}(\Omega).
\end{equation*}
If \ref{p:m} is in force, then
\begin{equation*}
\lstort^{\Phi,N}_{1,1}(\Omega)=H^{\Phi,N}(\Omega).
\end{equation*}
Moreover, any $\Phi$-Cheeger $N$-cluster of $\Omega$ is also a $(1,\Phi)$-eigen-$N$-cluster of $\Omega$.
\end{proposition}

\begin{proof}
Given any $N$-cluster $\nupla E$ of $\Omega$, each $\nupla E_i$ has positive measure and finite perimeter by \cref{def:cluster}, and finite measure since $\nupla E_i \subset \Omega$. 
Hence $\nupla E_i$ admits a Cheeger set $\tilde{\nupla E}_i$, see~\cite{FPSS24}*{Sect.~3.1},  i.e., 
\begin{equation}
\label{eq:toast}
h(\nupla E_i)
=
\frac{\per(\tilde{\nupla E}_i)}{|\tilde{\nupla E}_i|},
\quad
\text{for each}\
i=1,\dots,N.
\end{equation}
Note that $\tilde{\nupla E}=(\tilde{\nupla E}_1,\dots,\tilde{\nupla E}_N)$ is an $N$-cluster of $\Omega$ such that
\begin{equation*}
{h(\nupla E)}
=
{\frac{\per(\tilde{\nupla E})}{|\tilde{\nupla E}|}}
\end{equation*}
by~\eqref{eq:toast}, hence proving that $H^{\Phi,N}(\Omega)\le\lstort^{\Phi,N}_{1,1}(\Omega)$.

Viceversa, we clearly have ${h(\nupla E)}\le\per(\nupla E)/|\nupla E|$ for any $N$-cluster $\E$ of $\Omega$.
By~\ref{p:m}, we hence get that 
\begin{equation}\label{eq:Phi(h)<Phi(PV)}
\Phi({h(\nupla E)}) 
\le 
\Phi\left({\frac{\per(\nupla E)}{|\nupla E|}}\right)
\end{equation}
for any $N$-cluster $\E$ of $\Omega$, yielding $\lstort^{\Phi,N}_{1,1}(\Omega)\le H^{\Phi,N}(\Omega)$. 

Finally, if $\nupla E$ is a $\Phi$-Cheeger $N$-cluster of $\Omega$, then~\eqref{eq:Phi(h)<Phi(PV)} yields that
\begin{equation*}
\lstort^{\Phi,N}_{1,1}(\Omega) 
\le 
\Phi({h(\nupla E)}) 
\le 
\Phi\left({\frac{\per(\nupla E)}{|\nupla E|}}\right) 
= 
H^{\Phi,N}(\Omega)
=
\lstort^{\Phi,N}_{1,1}(\Omega) ,
\end{equation*}
and thus $\E$ must also be a $(1,\Phi)$-eigen-$N$-cluster of $\Omega$, concluding the proof.
\end{proof}

The second part of \cref{res:H=lstort} cannot be reversed, that is, $(1,\Phi)$-eigen-$N$-clusters of $\Omega$ may not be $\Phi$-Cheeger $N$-clusters of $\Omega$, see also~\cite{BP18}.
However, this holds in the case of $1$-adjusted $N$-clusters (recall \cref{def:1-adj}). 
Precisely, we have the following result. 

\begin{proposition}
\label{res:1_adj_min_IFF}
Let \ref{p:m} be in force.
Then, any $1$-adjusted $(1,\Phi)$-eigen-$N$-cluster of $\Omega$ is also a $\Phi$-Cheeger $N$-cluster of $\Omega$.
\end{proposition}

\begin{proof}
If $\E$ is a $1$-adjusted $(1,\Phi)$-eigen-$N$-cluster of $\Omega$, then, by~\eqref{eq:1-adj_self}, 
\begin{equation*}
\Phi\left({\frac{\per(\nupla E)}{|\nupla E|}}\right)
=
\Phi({h(\nupla E)})
=
\lstort^{\Phi,N}_{1,1}(\Omega).
\end{equation*}
Hence the conclusion immediately follows from \cref{res:H=lstort}.
\end{proof}

It is worth noting that, whenever the strict monotonicity property~\ref{p:ms} holds, any minimizer of~\eqref{eq:H} is $1$-adjusted. 

\begin{proposition}
\label{res:4_1-adj}
Let \ref{p:ms} be in force.
Then, any $\Phi$-Cheeger $N$-cluster of $\Omega$ is $1$-adjusted.
\end{proposition}

\begin{proof}
By contradiction, if $\nupla E$ is a $\Phi$-Cheeger $N$-cluster of $\Omega$ which is not $1$-adjusted, then
\begin{equation*}
\frac{\per(\nupla E_i)}{|\nupla E_i|} > h(\Omega\setminus\Omega^{\nupla E}_i)
\quad
\text{for some}\
i\in\set*{1,\dots,N},
\end{equation*}
where $\Omega^{\nupla E}_i$ is defined as in~\eqref{eq:def_M_i}. Since $|\Omega\setminus \Omega^{\nupla E}_i|<\infty$ and $\nupla E_i\subset\Omega\setminus\Omega^{\nupla E}_i$, by standard results (e.g., see~\cite{FPSS24}*{Sect.~3.1}) the set $\Omega\setminus\Omega^{\nupla E}_i$ admits a Cheeger set $\tilde{\nupla E}_i$, i.e., a set such that
\begin{equation*}
h(\Omega\setminus\Omega^{\nupla E}_i) 
= 
\frac{\per(\tilde{\nupla E}_i)}{|\tilde{\nupla E}_i|}.
\end{equation*}
Assuming $i=1$ without loss of generality, the $N$-cluster 
$\tilde{\nupla E} = 
(\tilde{\nupla E}_1, \nupla E_2,\dots,\nupla E_N)$
satisfies ${\per(\tilde{\nupla E})/|\tilde{\nupla E}|}<{\per(\nupla E)/|\nupla E|}$, therefore property~\ref{p:ms} implies the validity of the strict inequality $\Phi({\per(\tilde{\nupla E})/|\tilde{\nupla E}|})<\Phi({\per(\nupla E)/|\nupla E|})$, against the minimality of~$\E$.
\end{proof}

Summing up these results, and, in view of \cref{res:clusters_only}, restricting the class of competitors for $\lstort^{\Phi, N}_{1,1}(\Omega)$ to $N$-clusters only, for a non-empty, bounded, and open set $\Omega$ we have the following chain of inclusions
\begin{align*}
\set*{\E\in\argmin \lstort^{\Phi,N}_{1,1}(\Omega)} 
&\supseteq 
\set*{\E\in\argmin H^{\Phi,N}(\Omega)}
\\
&\supseteq 
\set*{\E\in\argmin H^{\Phi,N}(\Omega): \E\ \text{is $1$-adjusted}}
\\
&= 
\set*{\E\in\argmin \lstort^{\Phi,N}_{1,1}(\Omega): \E\ \text{is $1$-adjusted}},
\end{align*}
and if \ref{p:ms} holds, the last set inclusion becomes a set equality.

\begin{remark}
To rephrase the results of the present subsection in the abstract setting of~\cite{FPSS24}, we just need to invoke~\cite{FPSS24}*{Th.~3.6}, and hence we need to enforce that the perimeter-measure pair meets properties (P.4), (P.5), and (P.6) of~\cite{FPSS24}*{Sect.~2.1}.
\end{remark}

Finally, assuming~\ref{p:c}, we can prove the following lower bound on $H^{\Phi,N}(\Omega)$, generalizing~\cite{P10}*{Prop.~3.14}.

\begin{proposition}
\label{res:lower_bound_H}
Let \ref{p:c} be in force. Then,
\begin{equation}
\label{eq:lower_bound_H}
H^{\Phi,N}(\Omega)
\ge
N \delta d\left(\frac{|B_1|}{|\Omega|}\right)^{\frac1d}, 
\end{equation}
holds, where $\delta$ is as in~\ref{p:c}.
\end{proposition}

\begin{proof}
For any $\eps>0$, we let $\nupla E^\eps$ be an $N$-cluster of $\Omega$ such that
\[
H^{\Phi, N}(\Omega) +\eps \ge \Phi\left(\frac{\per(\nupla E^\eps)}{|\nupla E^\eps|} \right).
\]
By~\ref{p:c}, the isoperimetric inequality (i.e., $\per(A)\ge d |B_1|^{\frac{1}{d}}|A|^{\frac{d-1}{d}}$ for all Borel sets $A\subset\R^d$) on each chamber $\E^\eps_i$, and the set inclusion $\E^\eps_i \subset \Omega$, we have
\begin{equation*}
H^{\Phi, N}(\Omega) +\eps 
\ge 
\delta \sum_{i=1}^N \frac{\per(\nupla E^\eps_i)}{|\nupla E^\eps_i|}
\ge
\delta d \sum_{i=1}^N \left(\frac{|B_1|}{|\E^\eps_i|}\right)^{\frac 1d}
\ge
N\delta d \left(\frac{|B_1|}{|\Omega|}\right)^{\frac 1d},
\end{equation*}
and the conclusion follows by letting $\eps\to 0^+$.
\end{proof}

\begin{remark} 
It is worth noticing that \cref{res:lower_bound_H} yields that $H^{\Phi,N}(\Omega)\to\infty$ as $N\to\infty$, generalizing~\cite{P10}*{Cor.~3.15}.     
\end{remark}

\subsection{Existence of minimizers}
\label{ssec:ex_cheeger_clus}

We now prove that $1$-adjusted minimizers of~\eqref{eq:H} exist among $N$-clusters of a non-empty, bounded, and open set $\Omega$, assuming~\ref{p:lsc}--\ref{p:m}. 
In virtue of \cref{res:H=lstort}, this also implies the existence of minimizers of~\eqref{eq:lstort_1}, generalizing the corresponding results in~\cites{P10,BP18,C17,CL19}. Note that \ref{p:c} here plays a crucial role, as it yields a uniform upper bound on the perimeters of an infimizing sequence.

\begin{theorem}[Existence of minimizers of $H^{\Phi,N}(\Omega)$]
\label{res:existence}
Let \ref{p:lsc}, \ref{p:c}, and \ref{p:m} be in force.
Then, 
$\Phi$-Cheeger $N$-clusters of $\Omega$ exist.
\end{theorem}

\begin{proof}
Let $\set*{\nupla E^k:k\in\N}$ be an infimizing sequence for $H^{\Phi,N}(\Omega)$ and let $\eps>0$.
By \ref{p:c}, for all $k\in\N$ sufficiently large we have that
\begin{equation*}
\sum_i \per(\nupla E^k_i)
\le 
\frac{|\Omega|}{\delta} \,\Phi\left(\frac{\per(\nupla E^k)}{|\nupla E^k|}\right)
\le 
\frac{|\Omega|}{\delta}\,(H^{\Phi,N}(\Omega) + \varepsilon),
\end{equation*}
where $\delta>0$ is as in \ref{p:c}.
Consequently, up to subsequences, $\nupla E^k_i\to\nupla E_i$ as $k\to\infty$ in $L^1(\Omega)$ for each $i=1,\dots,N$, for some $\nupla E_i\subset\Omega$.
By the lower semicontinuity of the perimeter, we have $\per(\nupla E_i)<\infty$, while it is also easy to see that $|\nupla E_i\cap\nupla E_j|=0$ for $i\ne j$ with $i,j\in\set*{1,\dots,N}$.
To conclude that $\nupla E$ is an $N$-cluster of $\Omega$, we need to check that $|\nupla E_i|>0$ for all $i=1,\dots,N$. 
If $|\nupla E_j|=0$ for some~$j$, then, thanks to~\ref{p:c} and the isoperimetric inequality, we can estimate
\begin{equation}
\label{eq:goccia}
\frac{H^{\Phi,N}(\Omega) + \eps}{\delta}
\ge 
\frac 1\delta \Phi\left(\frac{\per(\nupla E^k)}{|\nupla E^k|}\right)
\ge
\sum_{i=1}^N 
\frac{\per(\nupla E^k_i)}{|\nupla E^k_i|} 
\ge 
\frac{\per(\nupla E^k_j)}{|\nupla E^k_j|} 
\ge 
\frac{\per(B^k_j)}{|B^k_j|} = 
\frac{d}{r_k},
\end{equation}
where $B^k_j\subset\R^d$ is any ball of radius $r_k>0$ such that $|B^k_j|=|\nupla E^k_j|$. 
Since $|\nupla E_j|=0$,  $|\nupla E^k_j|\to0^+$ as $k\to\infty$, and thus also $r_k\to0^+$ as $k\to\infty$, contradicting~\eqref{eq:goccia}. 
Therefore, $\E$ is an $N$-cluster such that, by the lower semicontinuity of the perimeter, 
\begin{equation*}
{\frac{\per(\nupla E)}{|\nupla E|}}
\le 
\liminf_{k\to\infty} {\frac{\per(\nupla E^k)}{|\nupla E^k|}}.
\end{equation*}
Now, owing to~\ref{p:m} and to~\ref{p:lsc}, we get that
\begin{equation*}
\Phi\left({\frac{\per(\nupla E)}{|\nupla E|}}\right)
\le 
\Phi\left(
\liminf_{k\to\infty}{\frac{\per(\nupla E^k)}{|\nupla E^k|}}\right) \le 
\liminf_{k\to\infty} \Phi\left({\frac{\per(\nupla E^k)}{|\nupla E^k|}}\right) 
= 
H^{\Phi,N}(\Omega),
\end{equation*}
yielding that $\E$ is a $\Phi$-Cheeger-$N$-cluster of $\Omega$ and concluding the proof.
\end{proof}

From \cref{res:H=lstort,res:existence} we  immediately get the following result.

\begin{corollary}[Existence of minimizers of $\lstort^{\Phi,N}_{1,1}(\Omega)$] 
\label{cor:existence_geom_1}
Let \ref{p:lsc}, \ref{p:c}, and \ref{p:m} be in force. 
Then, $(1,\Phi)$-eigen-$N$-clusters of $\Omega$ exist.
\end{corollary}

\begin{remark}[More general version of \cref{res:existence}]
\label{rem:existence_weak}
The assumptions on $\Omega$ yielding the validity of \cref{res:existence} can be considerably weakened.
In fact, it is enough to assume that $\Omega\subset\R^d$ is a measurable set with $|\Omega|\in(0,\infty)$ containing at least one viable competitor.
We omit the proof of this statement (also compare with the general approach of~\cite{FPSS24}).
\end{remark}

\subsection{Properties of minimizers}

Let us collect some basic yet quite useful properties of $\Phi$-Cheeger $N$-clusters, i.e., minimizers of~\eqref{eq:H}.

\begin{proposition}[Properties of $\Phi$-Cheeger $N$-clusters]
\label{res:props_cheeger}
If $\nupla E$ is a $\Phi$-Cheeger $N$-cluster of $\Omega$, then:
\begin{enumerate}[label=(\roman*)]

\item 
\label{item:prop_cheeger_meas}
enforcing \ref{p:c}, 
the following uniform lower bound
\begin{equation}
\label{eq:prop_cheeger_meas}
|\nupla E_i|
\ge 
|B_1| \left(\frac{\delta d}{H^{\Phi,N}(\Omega)}\right)^d,
\quad
\text{for}\
i=1,\dots,N,
\end{equation}
holds, where $\delta>0$ is as in \ref{p:c};

\item
\label{item:prop_cheeger_modified_1-adj}
enforcing \ref{p:m}, 
$\nupla E$ can be modified into a $1$-adjusted $\Phi$-Cheeger $N$-cluster.

\end{enumerate}
\end{proposition}

\begin{proof}
We prove each statement separately.

\vspace{1ex}

\textit{Proof of \ref{item:prop_cheeger_meas}}.
The proof is essentially the same as that of~\cref{res:lower_bound_H}, the only difference being that one works with a minimizer $\nupla E$. We omit the simple details.

\vspace{1ex}

\textit{Proof of \ref{item:prop_cheeger_modified_1-adj}}.
The proof is quite similar to the first part of the proof of \cref{res:4_1-adj}.
By standard results (e.g., see~\cite{FPSS24}*{Sect.~3.1}), $\Omega\setminus\Omega^{\nupla E}_1$ admits a Cheeger set $\tilde{\nupla E}_1$, being $\Omega^{\nupla E}_i$ defined as in~\eqref{eq:def_M_i}. 
Consequently,
\begin{equation*}
\frac{\per(\tilde{\nupla E}_1)}{|\tilde{\nupla E_1}|} 
\le 
\frac{\per(A)}{|A|} 
\quad 
\text{for any}\
A\subset\Omega\setminus\Omega^{\nupla E}_1\
\text{such that}\ |A|>0.
\end{equation*}
As this holds also for $A=\nupla E_1$, we get that 
\begin{equation*}
h(\tilde{\nupla E}_1, \nupla E_2, \dots,\nupla E_N)\le {h(\nupla E)}.
\end{equation*}
By \ref{p:m}, the $N$-cluster $(\tilde{\nupla E}_1, \nupla E_2,\dots,\nupla E_N)$ is a $\Phi$-Cheeger $N$-cluster of $\Omega$. 
Repeating this procedure $N-1$ times on the remaining indexes gives the desired $1$-adjusted $\Phi$-Cheeger $N$-cluster of $\Omega$ and concludes the proof.
\end{proof}

\begin{remark}
To restate the results of the present subsection in the abstract setting of~\cite{FPSS24}, the perimeter-measure pair must satisfy properties (P.4), (P.5), and (P.6) in~\cite{FPSS24}*{Sect.~2.1} (this also ensures the validity of~\cite{FPSS24}*{Th.~3.6}, needed in the proof of \cref{res:existence}).
Note that the lower bound~\eqref{eq:prop_cheeger_meas} (and, consequently, also the one in \eqref{eq:lower_bound_H}) requires a finer version of the \emph{isoperimetric property} (P.6) of~\cite{FPSS24}*{Sect.~2.1}, see, e.g.,~\cite{FPSS24}*{Prop.~7.2} in the context of metric-measure spaces and the discussion in~\cite{FPSS24}*{Sect.~7.3} for non-local perimeter functionals.
\end{remark}

\subsection{Regularity of \texorpdfstring{$1$}{1}-adjusted minimizers}
\label{subsec:regularity}

We now establish the regularity of $1$-adjusted minimizers of~\eqref{eq:H}, assuming~\ref{p:c}.
We adapt~\cite{BP18}*{Sect.~3}, where the authors deal with $1$-adjusted minimizers of~\eqref{eq:lstort_1} for the choice $\Phi=\|\cdot\|_\infty$. 
We omit the full proofs and only detail the minor changes. To start, we recall the following two standard definitions.

\begin{definition}[Mean curvature bounded from above]\label{defin:mc_bound}
A set $F\subset\Omega$ has \emph{distributional mean curvature bounded from above} at scale $r_0\in(0,\infty]$ by $g\in L^1_\loc(\Omega)$ in $\Omega$ if 
\begin{equation*}
\per(F;B_r(x)) 
\le 
\per(E;B_r(x)) + 
\int_{F\setminus E} g\de y
\end{equation*}
whenever $B_r(x)\Subset\Omega$ with $x\in\R^d$, $r\in(0,r_0)$, and $E\subset F$ with $F\setminus E\Subset B_r(x)$.
\end{definition}

\begin{definition}[$(\Lambda,r_0)$-minimizer of the perimeter]\label{def:L_minimizer}
A set $F\subset\R^d$ is a \emph{$(\Lambda,r_0)$-minimizer of the perimeter} in $\Omega$, with $\Lambda<\infty$ and $r_0\in[0,\infty]$, if
\begin{equation*}
\per(F;B_r(x))\le \per(E;B_r(x)) + \Lambda|E\Delta F|
\end{equation*}
whenever $E\subset\R^d$ is such that $E\Delta F\Subset B_r(x) \cap \Omega$ with $x\in\R^d$, $r\in(0,r_0)$.
\end{definition}

The following two results give curvature bounds for $1$-adjusted minimizers of~\eqref{eq:H}, inside a non-empty, bounded, and open set $\Omega$, assuming~\ref{p:c}.

\begin{lemma}[Curvature bound, I]
Let property~\ref{p:c} be in force.
If $\nupla E$ is a $1$-adjusted $\Phi$-Cheeger $N$-cluster of $\Omega$, then the sets $\Omega^{\nupla E}_i$ defined in~\eqref{eq:def_M_i}, $i=1,\dots,N$, have distributional mean curvature bounded from above at scale $r_0=\delta d(H^{\Phi,N}(\Omega))^{-1}$ by $H^{\Phi,N}(\Omega)\delta^{-1}$ in $\Omega$, where $\delta>0$ is as in~\ref{p:c}.
\end{lemma}

\begin{proof}
The proof goes as that of~\cite{BP18}*{Lem.~3.3}. 
The first part of the argument requires the choice $r_0=\delta d(H^{\Phi,N}(\Omega))^{-1}$ and~\ref{p:c}. 
For the second part of the argument, to achieve the upper bound on the curvature, it is enough to observe that
\begin{equation}\label{eq:change_from_BP}
\frac{\per(\nupla E_i)}{|\nupla E_i|} 
= 
h(\nupla E_i) 
\le 
\|{h(\nupla E)}\|_1  
\le 
\frac{\Phi({h(\E)})}{\delta} 
= 
\frac{H^{\Phi,N}(\Omega)}{\delta},
\end{equation}
owing to the $1$-adjusted hypothesis on the cluster, \cref{rem:1-adj_self} (see~\eqref{eq:1-adj_self}), and to~\ref{p:c}.
We leave the simple details to the reader.
\end{proof}

\begin{lemma}[Curvature bound, II]
Let property~\ref{p:c} be in force.
If $\E$ is a $1$-adjusted $\Phi$-Cheeger $N$-cluster of $\Omega$, then each chamber $\nupla E_i$ has distributional mean curvature bounded from above at scale $r_0=\delta d(H^{\Phi,N}(\Omega))^{-1}$ by $H^{\Phi,N}(\Omega)\delta^{-1}$ in $\Omega$, where $\delta>0$ is as in~\ref{p:c}.
\end{lemma}

\begin{proof}
As $\nupla E$ is $1$-adjusted, by \cref{rem:1-adj_self} each chamber $\nupla E_i$ is a Cheeger set of $\Omega\setminus\Omega^{\nupla E}_i$ for every  $i=1,\dots,N$. 
Therefore, by standard results (e.g., see~\cite{LNS17}*{Lem.~2.2}), $\nupla E_i$ has distributional mean curvature bounded from above at scale $r_0=\delta d (H^{\Phi,N}(\Omega))^{-1}$ by $h(\Omega\setminus\Omega^{\nupla E}_i)$ in $\Omega$.
Recalling~\eqref{eq:1_adj_h_chamber=complement} and using~\ref{p:c}, we have
\begin{equation*}
h(\Omega\setminus\Omega^{\nupla E}_i) 
= 
h(\nupla E_i) 
\le 
\|{h(\nupla E)}\|_1 
\le 
\frac{\Phi({h(\nupla E)})}{\delta}
= 
\frac{H^{\Phi,N}(\Omega)}{\delta},
\end{equation*}
so that the conclusion follows by noticing that, by \cref{defin:mc_bound}, if $c_1>0$ is a bound from above to the distributional curvature, so it is any $c_2>c_1$.
\end{proof}

The following result states that all chambers of a $1$-adjusted minimizer of~\eqref{eq:H} on a non-empty, bounded, and open set $\Omega$ are almost minimizers of the perimeter in the sense of \cref{def:L_minimizer}.

\begin{lemma}[Almost minimizer]
\label{res:chamber_almost_min}
Let \ref{p:c} be in force.
If $\nupla E$ is a $1$-adjusted $\Phi$-Cheeger $N$-cluster of $\Omega$, then each chamber $\nupla E_i$ is a $(\Lambda,r_0)$-minimizer of the perimeter in $\Omega$, with $\Lambda = H^{\Phi,N}(\Omega)\delta^{-1}$ and $r_0=\delta d (H^{\Phi,N}(\Omega))^{-1}$, where  $\delta>0$ is as in~\ref{p:c}.
\end{lemma}

\begin{proof}
The proof goes as that of~\cite{BP18}*{Prop.~3.4}. The only relevant change is to use \ref{p:c} to get an estimate similar to~\eqref{eq:change_from_BP}.
We leave the simple details to the reader.
\end{proof}

In virtue of the standard theory for almost minimizers of the perimeter (refer to~\cite{Mag12book} for an account), we get the following regularity properties for $1$-adjusted minimizers of~\eqref{eq:H} on a non-empty, bounded, and open set $\Omega$. 

\begin{theorem}[Regularity]
\label{res:regularity}
Let \ref{p:c} be in force.
If $\nupla E$ is a $1$-adjusted $\Phi$-Cheeger $N$-cluster of $\Omega$, then the following hold true:
\begin{enumerate}[label=(\roman*),leftmargin=5ex]

\item\label{item:reg_gamma} 
each $\partial^*\nupla E_i\cap \Omega$ is of class $C^{1,\gamma}$ for every $\gamma\in(0,\sfrac12)$;

\item\label{item:reg_haus_dim} 
each $\partial \nupla E_i \setminus \partial^* \nupla E_i$ has Hausdorff dimension at most $d-8$;

\item\label{item:reg_dim7}  
if $d\le 7$, then each $\partial \nupla E_i$ is of class $C^{1,\gamma}$ for every $\gamma\in(0,\sfrac12)$;

\item\label{item:reg_equiv} 
if $\mathscr{H}^{d-1}(\partial \Omega)<\infty$, then there exists a $1$-adjusted $\Phi$-Cheeger $N$-cluster $\tilde{\nupla E}$ of $\Omega$ such that $|\tilde{\nupla E}_i\bigtriangleup\nupla E_i|=0$ and each $\tilde{\nupla E_i}$ is open;

\item\label{item:reg_tang}  
if $\per(\Omega)<\infty$, then each $\partial^*\nupla E_i\cap\Omega$ can meet $\partial^*\Omega$ only tangentially, i.e., if $x\in\partial\nupla E_i \cap \partial^*\Omega$, then $x\in\partial^*\nupla E_i$ and $\nu_\Omega(x)=\nu_{\nupla E_i}(x)$.
\end{enumerate}
\end{theorem}

\begin{proof}
Due to \cref{res:chamber_almost_min}, properties~\ref{item:reg_gamma} and~\ref{item:reg_haus_dim} follow from the regularity theory of almost minimizers, see~\cite{Mag12book}*{Ths.~21.8 and 28.1}.
Property~\ref{item:reg_dim7} is an immediate consequence of~\ref{item:reg_gamma} and~\ref{item:reg_haus_dim}.
For property~\ref{item:reg_equiv}, it is enough to set $\tilde{\nupla E}_i=\nupla E_i\setminus\partial\nupla E_i$ for $i=1,\dots,N$ (see the proof of~\cite{BP18}*{Th.~3.5}).
Finally, property~\ref{item:reg_tang} can be proved as in~\cite{LP16}*{App.~A} or as in~\cite{LS18}*{Th.~3.5}.
\end{proof}

Finally, owing to \cref{res:regularity} and to~\cite{S15}*{Th.~1.1} (see also~\cites{CT17,GHL23}), one can approximate the chambers of a $1$-adjusted minimizing cluster from within the interior with smooth sets, both in $L^1$ and in perimeter, provided that $\Omega$ is sufficiently regular.

\begin{corollary}[Approximation]
\label{res:approx}
Assume that $\per (\Omega)<\infty$ and $\mathscr H^{d-1}(\partial\Omega\setminus\partial^*\Omega)=0$. 
If $\nupla E$ is a $1$-adjusted $\Phi$-Cheeger $N$-cluster of $\Omega$ such that each chamber $\nupla E_i$ is open, then there exist $N$-clusters $\set*{\nupla E^k : k\in\N}$ of $\Omega$ such that
$\nupla E^k_i\Subset\nupla E_i$,
$\partial\nupla E^k_i$
is smooth for all $k\in\N$,
$\nupla E^k_i\to\nupla E_i$ in  $L^1(\Omega)$ and 
$\per(\nupla E^k_i)\to\per(\nupla E_i)$ as $k\to\infty$, for each $i=1,\dots,N$.
\end{corollary}

\begin{proof}
The proof is identical to that of~\cite{BP18}*{Prop.~3.6} and so we omit it.
\end{proof}

\section{Relation with the functional problem}
\label{sec:caroccia-littig}

In this section, we introduce the functional variant of~\eqref{eq:lstort_1}, and we shall see how it is related to~\eqref{eq:H}.
We adapt~\cite{CL19}*{Sect.~3} (see also~\cite{FPSS24}*{Sect.~5}), where the authors deal with $\Phi=\|\cdot\|_1$, omitting the full proofs and only detailing the minor changes.

\subsection{\texorpdfstring{$BV_0$}{BV0} space and the relation \texorpdfstring{$h=\lambda_{1,1}$}{h=lambda 1,1}}

We start with the following definition of $BV_0$ space, which we will use in the remainder of the paper.

\begin{definition}[$BV_0$ space]
\label{def:BV_0}
Given a set $F\subset\R^d$, we let 
\begin{equation}
\label{eq:BV_0}
BV_0(F)
=
\set*{u\in BV(\R^d) : u=0\ \text{a.e.\ in}\ \R^d\setminus F},
\end{equation}
and we let $\mathcal u\in BV_0(F;\R^N)$ if $\nupla u_i\in BV_0(F)$ for  $i=1,\dots,N$.
\end{definition}

\begin{remark}
\label{rem:BV_emb}
Note that $BV_0(F)$ may not coincide with the space of $BV$ functions on $F$ with null trace at the boundary, unless $\partial F$ is sufficiently regular, see~\cite{CL19}*{Rem.~1.1}. Nevertheless, the usual Sobolev embeddings hold on a bounded $F$, as $BV_0(F)\subset BV_0(B_R)$ with $R>0$ such that $F \Subset B_R$. 
\end{remark}

We now introduce the usual, variational definition of first $1$-eigenvalue.

\begin{definition}[First $1$-eigenvalue]
\label{def:lambda_1}
The \emph{first $1$-eigenvalue} of a set $F\subset \R^d$ is
\begin{equation}
\label{eq:lambda_1}
\lambda_{1,1}(F)
=
\inf\set*{|Du|(\R^d) : u\in BV_0(F),\ \|u\|_{L^1}=1}\in[0,\infty].
\end{equation}
\end{definition}

\begin{remark}[Non-negative competitors]
\label{rem:restriction_pos_comp_l1}
The competitors in~\eqref{eq:lambda_1} can be chosen non-neg\-a\-tive. Indeed, by the chain rule, if $u\in BV_0(F)$, then also $|u|$ belongs to $BV_0(F)$ with $|D|u||(\R^d)=|Du|(\R^d)$.
\end{remark}

We recall the following standard result, relating the Cheeger constant of a set $F$, $h(F)$, to the first Dirichlet eigenvalue of the $1$-Laplacian on the set $F$, $\lambda_{1,1}(F)$, refer to~\cite{FPSS24}*{Th.~5.4} (refer also to~\cite{CL19}*{Prop.~2.1}). We remark that, in the given references, it is assumed that $F$ has positive finite measure and contains at least one $N$-cluster, but this is not necessary, and the proof can be repeated almost \emph{verbatim}.

\begin{theorem}[$h=\lambda_{1,1}$]
\label{res:h=lambda_1}
Given a set $F\subset\R^d$, we have $h(F)=\lambda_{1,1}(F)$.
\end{theorem}

As a simple yet quite useful consequence of \cref{res:h=lambda_1}, we get the following result.

\begin{corollary}
\label{res:lstort_1_N-set}
Given a set $F\subset\R^d$, it holds that
\begin{equation*}
\lstort^{\Phi,N}_{1,1}(F)
=
\inf\set*{
\Phi({\lambda_{1,1}(\nupla E)}) : \nupla E\ \text{is an $N$-set of $F$}
}\in[0,\infty].
\end{equation*}    
\end{corollary}

\begin{remark}
\label{rem:ex_eigenfunctions_1}
Note that, by its very definition, $BV_0(F)\ne\set*{0}$ if and only if $\lambda_{1,1}(F)<\infty$. If, in addition, $F$ is bounded, then there exist eigenfunctions, that is, functions $u\in BV_0(F)$ realizing the infimum in~\eqref{eq:lambda_1}. To see this, it is enough to take an infimizing sequence, to use the compact embeddings (refer to \cref{rem:BV_emb}), and to exploit the lower semicontinuity of the total variation. 
Notice that, in virtue of \cref{rem:restriction_pos_comp_l1}, we can also assume these to be non-negative. In particular, this holds true for any non-empty, bounded, and open set $\Omega$. Moreover, if $\partial\Omega$ is sufficiently regular, then $\lambda_{1,1}(\Omega)$ is the usual first eigenvalue of the $1$-Laplacian on $\Omega$.
\end{remark}

\begin{remark}
The present subsection can be rephrased almost \textit{verbatim} in the abstract setting of~\cite{FPSS24} (in particular, see~\cite{FPSS24}*{Sect.~5}) enforcing the validity of (P.1), (P.2), (P.4), and (P.7), the latter ensuring the validity of \cref{rem:restriction_pos_comp_l1}. 
\end{remark}

\subsection{First \texorpdfstring{$1$}{1}-functional eigenvalue}\label{ssec:1_functional_eigenvalue}

We provide an analog of \cref{res:h=lambda_1} for the more general problem~\eqref{eq:H}.
We begin with the following definition, introducing our class of competitors, in the same spirit of~\cite{CL19}. 

\begin{definition}[$(1,N)$-function]
\label{def:1N-function}
We say that $\nupla u\in BV_0(F;\R^N)$ is a \emph{$(1,N)$-function} of $F\subset\R^d$ if $\nupla u_i\ge0$, $\|\nupla u_i\|_{L^1}=1$ and $\nupla u_i\,\nupla u_j=0$ a.e.\ in $F$ whenever $i\neq j$, for $i,j=1,\dots,N$.
\end{definition}

Note that any $N$-cluster $\nupla E$ of $F$ naturally induces a $(1,N)$-function $\nupla u^\nupla E$ of $F$, by letting 
\begin{equation}\label{eq:cluster_to_function}
\nupla u^\nupla E
=
\left(\frac{\chi_{\nupla E_1}}{|\nupla E_1|},\dots,\frac{\chi_{\nupla E_N}}{|\nupla E_N|}\right).
\end{equation}

The following definition was given in~\cite{CL19}*{eq.~(7)} for the special case $\Phi=\|\cdot\|_1$.

\begin{definition}[First $1$-functional $(\Phi,N)$-eigenvalue]\label{def:Lambda_1}
The \emph{first $1$-functional $(\Phi,N)$-ei\-gen\-va\-lue} of a set $F\subset\R^d$ is 
\begin{equation}
\label{eq:Lambda_1}    
\Lambda^{\Phi,N}_{1,1}(F)
=
\inf\set*{\Phi({[\nupla u]_{1,F}}) : \text{$\nupla u$ is a  $(1,N)$-function of $F$}}\in[0,\infty],
\end{equation}
where, for brevity, we have set
\begin{equation*}
{[\nupla u]_{1,F}}
=
\left(|D\nupla u_1|(\R^d),\dots, |D\nupla u_N|(\R^d)\right),
\end{equation*}
and, if no confusion can arise, we shall drop the reference to the ambient set $F$ and write $[\nupla u]_{1}$.
Any $(1,N)$-function $\nupla u$ of $F$ achieving the infimum is a \emph{$(1,\Phi)$-eigen-$N$-function} of $F$.
\end{definition} 

By~\eqref{eq:cluster_to_function}, given any set $F$ admitting an $N$-cluster, one has $\Lambda_{1,1}^{\Phi,N}(F)<\infty$, and also that $\Lambda_{1,1}^{\Phi,N}(F)\le H^{\Phi,N}(F)$. In particular, this holds  for any non-empty, bounded, open set~$\Omega$. 

\subsection{Existence of minimizers of \texorpdfstring{$\Lambda^{\Phi,N}_{1,1}(\Omega)$}{1-functional eigenvalue}}

Similarly to \cref{ssec:ex_cheeger_clus}, we show that there exist minimizers of the spectral-functional eigenvalue $\Lambda^{\Phi,N}_{1,1}(\Omega)$, up to assuming \ref{p:lsc}--\ref{p:m}. Once again \ref{p:c} plays the crucial role of yielding a uniform bound on the sequence of total variations of an infimizing sequence. This result generalizes~\cite{CL19}*{Th.~3.1}.

\begin{theorem}[Existence of minimizers of $\Lambda^{\Phi,N}_{1,1}(\Omega)$]
\label{res:existence_f}
Let \ref{p:lsc}, \ref{p:c}, and \ref{p:m} be in force.
Then, $(1,\Phi)$-eigen-$N$-functions of $\Omega$ exist.
\end{theorem}

\begin{proof}
Let $\set*{\nupla u^k : k\in\N}$ be an infimizing sequence for $\Lambda_{1,1}^{\Phi,N}(\Omega)$ and let $\eps>0$.
By~\ref{p:c}, for all $k\in\N$ sufficiently large we have that
\begin{equation*}
\Lambda_{1,1}^{\Phi,N}(\Omega)+\eps 
\ge 
\Phi({[\nupla u^k]_{1}})
\ge 
\delta |D\nupla u^k_i|(\R^d),
\end{equation*}
where $\delta>0$ is as in \ref{p:c}.
Since $\Omega\subset\R^d$ is bounded, the embedding $BV_0(\Omega)\subset L^1(\Omega)$ is compact.
Thus, up to subsequences, $\nupla u^k_i\to\nupla u_i$ as $k\to\infty$ in $L^1(\Omega)$ for  $i=1,\dots,N$, for some  $\nupla u_i\in L^1(\Omega)$. 
It is easy to check that $\nupla u$ is a $(1,N)$-function of $\Omega$. Moreover, 
\begin{equation*}
\Phi({[\nupla u]_{1}})
\le 
\Phi\left(\liminf_{k\to\infty} 
{[\nupla u^k]_{1}}\right)
\le 
\liminf_{k\to\infty} 
\Phi\left({[\nupla u^k]_{1}}\right)
=\Lambda_{1,1}^{\Phi,N}(\Omega)
\end{equation*}
thanks to  the lower semicontinuity of the $BV$ seminorm, to~\ref{p:m}, and to~\ref{p:lsc}, readily yielding the conclusion.
\end{proof}

\begin{remark}[More general version of \cref{res:existence_f}]
\label{rem:existence_f}
Similarly to \cref{rem:existence_weak}, \cref{res:existence_f} holds under weaker assumptions on~$\Omega$.
In fact, it is enough to assume that $\Omega\subset\R^d$ is a bounded measurable set with $|\Omega|>0$ containing at least one viable competitor.
Note that the boundedness of $\Omega$ cannot be relaxed to $|\Omega|<\infty$, as this does not necessarily guarantee the compactness of the embedding $BV_0(\Omega)\subset L^1(\Omega)$.
For a more detailed discussion,  see~\cite{Maz11book}*{Sect.~9.1.7} which contains the sharp hypotheses to ensure the compactness of the embedding.
\end{remark}

\begin{remark}
In order to rephrase the content of this subsection in the abstract setting of~\cite{FPSS24}, we have, at least, to enforce properties (P.1), (P.2), and (P.4). Notice that the definition of $\Lambda^{\Phi, N}_{1,1}(\Omega)$ we are using here---that is, by considering only non-negative competitors---corresponds to the one appearing in~\cite{FPSS24}*{Rem.~5.9}. 
Enforcing (P.7) allows us to drop this restriction, thanks to \cref{rem:restriction_pos_comp_l1}.
Furthermore, in order to achieve \cref{res:existence_f}, we need to ensure the compactness of the embedding $BV_0(\Omega, \mathfrak{m})\subset L^1(\Omega, \mathfrak{m})$. 
Note that this holds in many of the frameworks discussed in~\cite{FPSS24}*{Sect.~7}. 
\end{remark}

\subsection{Relations with first \texorpdfstring{$1$}{1}-functional eigenvalue}\label{ssec:skeleton_CL}

In the following result we prove the equivalence of problems~\eqref{eq:lstort_1} and~\eqref{eq:Lambda_1} under the validity of \ref{p:m}.

\begin{theorem}[$\Lambda^{\Phi,N}_{1,1}(\Omega)=\lstort^{\Phi,N}_{1,1}(\Omega)$]
\label{res:lstort=Lambda_1}
The following holds
\begin{equation*}
\lstort^{\Phi,N}_{1,1}(\Omega) \ge \Lambda^{\Phi,N}_{1,1}(\Omega).
\end{equation*}
If \ref{p:m} is in force, then 
\begin{equation*}
\lstort^{\Phi,N}_{1,1}(\Omega)=\Lambda^{\Phi,N}_{1,1}(\Omega).
\end{equation*}
Moreover, if $\nupla u$ is a $(1,\Phi)$-eigen-$N$-function of $\Omega$, then 
\begin{equation}
\label{eq:function_to_set_1}
\nupla E
=
\big(\set*{\nupla u_1>0},\dots,\set*{\nupla u_N>0}\big)
\end{equation}
is a $(1,\Phi)$-eigen-$N$-set of $\Omega$. 
Viceversa, if $\E$ is a $(1,\Phi)$-eigen-$N$-set of $\Omega$, there exists a $(1,\Phi)$-eigen-$N$-function $\nupla u$ such that $\set*{\nupla u_i>0} \subset \E_i$ for all $i=1,\dots, N$.
\end{theorem}

\begin{proof}
Given $\eps>0$, we can find an $N$-set $\nupla E$ of $\Omega$ such that $\lstort^{\Phi,N}_{1,1}(\Omega)+\eps>\Phi(\lambda_{1,1}(\nupla E))$, with $\lambda_{1,1}(\nupla E)\in\ort$.
Since $\E_i$ is a subset of a bounded set $\Omega$, we get the existence of non-negative eigenfunctions of $\lambda_{1,1}(\E_i)$, as noted in \cref{rem:ex_eigenfunctions_1}.

For each $i\in\set*{1,\dots,N}$, we now let  $\nupla u_i\in BV_0(\nupla E_i)$ be such that $\nupla u_i\ge0$, $\|\nupla u_i\|_{L^1}=1$ and $|D\nupla u_i|(\R^d)=\lambda_{1,1}(\nupla E_i)$.
Hence $\nupla u=(\nupla u_1,\dots,\nupla u_N)$ is a $(1,N)$-function of $\Omega$ as in \cref{def:1N-function} such that $[\nupla u]_{1}=\lambda_{1,1}(\nupla E)$.
We thus get that $\lstort^{\Phi,N}_{1,1}(\Omega)+\eps>\Phi([\nupla u]_{1})\ge\Lambda^{\Phi,N}_{1,1}(\Omega)$.
The claim hence follows by letting $\eps\to 0$.

On the other hand, let $\nupla u$ be a $(1,N)$-function of $\Omega$.
Using \cref{def:1N-function}, it is easy to check that $\nupla E$ in~\eqref{eq:function_to_set_1} is an $N$-set of $\Omega$ as in \cref{def:cluster}.
Hence, recalling \cref{def:lambda_1}, in virtue of $\lambda_{1,1}(\nupla E)\le[\nupla u]_{1}$, we have that $\Phi(\lambda_{1,1}(\nupla E))\le\Phi([\nupla u]_{1})$ by \ref{p:m}, yielding that $\lstort_{1,1}^{\Phi,N}(\Omega)\le\Lambda_{1,1}^{\Phi,N}(\Omega)$.

For the second part of the statement, if $\nupla u$ is a $(1,\Phi)$-eigen-$N$-function of $\Omega$, then $\nupla E$ in~\eqref{eq:function_to_set_1} satisfies $\lambda_{1,1}(\nupla E)\le[\nupla u]_{1}$. 
By \ref{p:m}, it follows that 
\[
\lstort^{\Phi,N}_{1,1}(\Omega)\le\Phi(\lambda_{1,1}(\nupla E))\le\Phi([\nupla u]_{1})=\Lambda^{\Phi,N}_{1,1}(\Omega),
\]
yielding that $\nupla E$ is a $(1,\Phi)$-eigen-$N$-set of $\Omega$.

Now let $\nupla E$ be a $(1,\Phi)$-eigen-$N$-set. 
Hence, $\lambda_{1,1}(\nupla E)\in \ort$, that is, $\lambda_{1,1}(\nupla E_i)<\infty$ for all $i=1,\dots,N$. Since $\E_i\subset\Omega$ is a bounded set with $|\nupla E_i|>0$, by \cref{rem:restriction_pos_comp_l1,rem:ex_eigenfunctions_1}, for all $i=1,\dots, N$, there exists a function $\nupla u_i\in BV_0(\nupla E_i)$ such that $\|\nupla u_i\|_{L^1}=1$, $\nupla u_i\ge0$, and $\lambda_{1,1}(\nupla E_i) = |D\nupla u_i|(\R^d)$. Therefore, $\nupla u = (\nupla u_1, \dots, \nupla u_N)$ is a $(1,N)$-function of $\Omega$ such that
\[
\Lambda_{1,1}^{\Phi, N}(\Omega)
\le
\Phi([\nupla u]_1)
= 
\Phi(\lambda_{1,1}(\E)) 
= 
\lstort_{1,1}^{\Phi, N}(\Omega).
\]
From the first part of the statement it follows that $\nupla u$ is a $(1,\Phi)$-eigen-$N$-function of $\Omega$, and, by construction, $\set*{\nupla u_i>0}\subset \nupla E_i$ for all $i=1,\dots, N$.
\end{proof}

\begin{remark}
\cref{res:lstort=Lambda_1} yields that, up to possibly passing to a smaller $N$-subset, each chamber of a $(1,\Phi)$-eigen-$N$-set of $\Omega$ is the zero superlevel set of a $(1,\Phi)$-eigen-$N$-function of $\Omega$. 
Actually, if a chamber $\E_i$ is a Cheeger set of itself, then the set inclusion is an equality. In such a case, $\nupla u_i = \chi_{\E_i}/\|\chi_{\E_i}\|_{L^1}$ is a first eigenfunction of the $1$-Laplacian on $\E_i$ (e.g., see \cite{FPSS24}*{Cor.~5.5}). In particular, this happens for \emph{all chambers} whenever \ref{p:ms} holds, in virtue of \cref{res:4_1-adj} and of \cref{rem:1-adj_self}.
\end{remark}

We are ready to deal with the main result of this section, generalizing~\cite{CL19}*{Th.~3.3}: we shall prove that, assuming~\ref{p:m}, for a non-empty, bounded, and open set $\Omega$, the equality $\Lambda_{1,1}^{\Phi,N}(\Omega)=H^{\Phi,N}(\Omega)$ holds.
In fact, assuming the stronger \ref{p:ms}, minimizers of one problem are naturally related to minimizers of the other problem (if they exist).

\begin{theorem}[$\Lambda_{1,1}^{\Phi,N}(\Omega)=H^{\Phi,N}(\Omega)$]
\label{res:H=Lambda}
Let \ref{p:m} be in force. Then, 
\begin{equation}
\label{eq:H=Lambda}
\Lambda_{1,1}^{\Phi,N}(\Omega)=H^{\Phi,N}(\Omega).
\end{equation} 
Moreover, under the stronger \ref{p:ms}, any $(1,N)$-function $\nupla u$ is a $(1,\Phi)$-eigen-$N$-function of $\Omega$ if and only if, for a.e.\ $t_i > 0$ such that $|\set*{\nupla u_i > t_i}| > 0$, for $i = 1,\dots,N$,
\begin{equation}
\label{eq:function_to_cluster}
\nupla E
=
\big(\set*{\nupla u_1>t_1},\dots,\set*{\nupla u_N>t_N}\big)
\end{equation} 
is a $1$-adjusted $\Phi$-Cheeger $N$-cluster of $\Omega$. 
In particular, if $\nupla E$ is a $\Phi$-Cheeger $N$-cluster of $\Omega$, then $\nupla u$ in~\eqref{eq:cluster_to_function} 
is a $(1,\Phi)$-eigen-$N$-function of $\Omega$.
\end{theorem}

For the proof of the second part of the statement of \cref{res:H=Lambda}, we need the following result, which extends~\cite{CL19}*{Lem.~3.4}. 

\begin{lemma}
\label{res:choose_t_i}
Let~\ref{p:ms} be in force.
If $\nupla u$ is a $(1,\Phi)$-eigen-$N$-function of $\Omega$, then for almost every $t\in \R$ such that $|\set*{\nupla u_i>t}|>0$ one has
\begin{equation*}
\frac{\per(\set*{\nupla u_i>t})}{|\set*{\nupla u_i>t}|}
=
|D\nupla u_i|(\R^d)
\end{equation*}
for each $i=1,\dots,N$.
\end{lemma}

\begin{proof}
By~\ref{p:ms}, if $j\in\set*{1,\dots,N}$ and $\nupla u,\bar{\nupla u}\in BV(\Omega;\R^N)$ are such that $\Phi({[\nupla u]_{1}})\le\Phi({[\bar{\nupla u}]_{1}})$ and $\nupla u_i=\bar{\nupla u}_i$ for  $i\in\set*{1,\dots,N}\setminus\set*{j}$, then $|D\nupla u_j|(\R^d)\le|D\bar{\nupla u}_j|(\R^d)$, as in~\cite{CL19}*{eq.~(10)}. 
Hence the proof is similar to that of~\cite{CL19}*{Lem.~3.4}.
We omit the details.
\end{proof}

\begin{proof}[Proof of \cref{res:H=Lambda}]
The equality $\Lambda_{1,1}^{\Phi,N}(\Omega)=H^{\Phi,N}(\Omega)$ immediately follows by combining \cref{res:H=lstort} and \cref{res:lstort=Lambda_1}.
We can hence deal with the second part of the statement, assuming \ref{p:ms}. 
We argue as in the proof of~\cite{CL19}*{Th.~3.3}. 

On the one hand, let $\nupla u$ be a $(1,\Phi)$-eigen-$N$-function of $\Omega$ and let $t_i>0$ be such that  \cref{res:choose_t_i} applies to each $i=1,\dots,N$. 
Therefore, we have
\begin{equation*}
\Lambda_{1,1}^{\Phi,N}(\Omega)
=
\Phi({[\nupla u]_{1}})
=
\Phi\left({\frac{\per(\nupla E)}{|\nupla E|}}\right)
\ge
H^{\Phi,N}(\Omega),
\end{equation*}
being $\nupla E$ the $N$-cluster in~\eqref{eq:function_to_cluster}. 
Hence, in virtue of~\eqref{eq:H=Lambda}, $\nupla E$ is a $\Phi$-Cheeger $N$-cluster of $\Omega$, which is $1$-adjusted thanks to \cref{res:4_1-adj}.

On the other hand, let $\nupla u$ be a $(1,N)$-function of $\Omega$ such that, for almost every $t_i>0$ with $|\{\nupla u_i>t_i\}|>0$, $\nupla E$ as in~\eqref{eq:function_to_cluster} is a $1$-adjusted $\Phi$-Cheeger $N$-cluster of $\Omega$. 
Now let $t_2,\dots,t_N$ be such that $|\set*{\nupla u_i>t_i}|>0$ for $i=2,\dots,N$, and set $\mathcal T=\set*{t>0: |\set*{u_1>t}|>0}$. 
We claim that  
\begin{equation}
\label{eq:trespolo}
t\mapsto\frac{\per(\set*{\nupla u_1>t})}{|\set*{\nupla u_1>t}|}
\quad
\text{is constant for}\ t\in\mathcal T.
\end{equation}
By contradiction, if this is not the case, we can find $t_1,\tau_1\in\mathcal T$, $t_1\ne\tau_1$, such that 
\begin{equation}\label{eq:pappagallo}
\frac{\per(\set*{\nupla u_1>t_1})}{|\set*{\nupla u_1>t_1}|} 
< 
\frac{\per(\set*{\nupla u_1>\tau_1})}{|\set*{\nupla u_1>\tau_1}|}.
\end{equation}
Accordingly to our hypotheses, also 
\begin{equation*}
\tilde{\nupla E} 
= 
(\set*{\nupla u_1>\tau_1}, \nupla E_2, \dots,\nupla E_N)
\end{equation*}
is a $1$-adjusted $\Phi$-Cheeger $N$-cluster of $\Omega$. Therefore, also owing to~\eqref{eq:1-adj_self}, we have
\begin{equation}\label{eq:contradiction}
\Phi({h(\nupla E)})
=
\Phi\left({\frac{\per(\nupla E)}{|\nupla E|}}\right)
=
H^{\Phi,N}(\Omega) 
=
\Phi\left({\frac{\per(\tilde{\nupla E})}{|\tilde{\nupla E}|}}\right)
=
\Phi({h(\tilde{\nupla E})}).
\end{equation}
Nevertheless, by~\eqref{eq:pappagallo}, we must have that
\begin{equation*}
h(\set*{\nupla u_1>t_1})
=
\frac{\per(\set*{\nupla u_1>t_1})}{|\set*{\nupla u_1>t_1}|} 
< 
\frac{\per(\set*{\nupla u_1>\tau_1})}{|\set*{\nupla u_1>\tau_1}|} 
= 
h(\set*{\nupla u_1>\tau_1}),
\end{equation*}
yielding ${h(\nupla E)} <{h(\tilde{\nupla E})}$.
By~\ref{p:ms}, it must be $\Phi({h(\nupla E)}) < \Phi({h(\tilde{\nupla E})})$, contradicting~\eqref{eq:contradiction}.
This concludes the proof of the claimed~\eqref{eq:trespolo}. 

Therefore, there exists $h_1>0$ (as a consequence of the isoperimetric inequality and of the fact that $|\set*{\nupla u_1>t}|>0$) such that
\begin{equation*}
\frac{\per(\set*{\nupla u_1>t})}{|\set*{\nupla u_1>t}|} = h_1
\quad 
\text{for all}\ t\in\mathcal T.
\end{equation*}
Reasoning analogously for each $i\in\set*{2,\dots,N}$, we find constants $h_i>0$ such that
\begin{equation*}
\frac{\per(\set*{\nupla u_i>t})}{|\set*{\nupla u_i>t}|} = h_i
\quad 
\text{for all}\ t>0\ \text{such that}\ 
|\set*{\nupla u_i>t}|>0.
\end{equation*}
Recalling that, by definition of $(1,N)$-function $\nupla u_i\ge 0$, and owing to the coarea formula, the above equalities, Cavalieri's principle (recalling that $\|\nupla u_i\|_{L^1}=1$), the validity of the equality $H^{\Phi,N}(\Omega)=\Phi(h_1,\dots,h_N)$, and~\eqref{eq:H=Lambda}, we have
\begin{align*}
\Lambda_{1,1}^{\Phi,N}(\Omega) 
&\le 
\Phi({[\nupla u]_{1}}) 
= 
\Phi\left(\int_0^{\infty} \per(\set*{\nupla u_1>t})\de t,\dots,\int_0^{\infty} \per(\set*{\nupla u_N>t})\de t\right)
\\
&= 
\Phi\left(h_1 \int_0^{\infty}|\set*{\nupla u_1>t}|\de t,\dots,h_N \int_0^{\infty}|\set*{\nupla u_N>t}|\de t\right)
\\
&=
\Phi\left(h_1,\dots,h_N\right) 
= 
H^{\Phi,N}(\Omega) 
= 
\Lambda_{1,1}^{\Phi,N}(\Omega),
\end{align*}
yielding that $\nupla u$ is a $(1,\Phi)$-eigen-$N$-function of $\Omega$.
The proof is complete.
\end{proof}

\begin{remark}
\label{rem:equality_1}
The equalities $\Lambda_{1,1}^{\Phi,N}(\Omega)=\lstort_{1,1}^{\Phi,N}(\Omega)=H^{\Phi,N}(\Omega)$ can be achieved within the abstract setting of~\cite{FPSS24}, enforcing properties (P.1), (P.2), (P.4), and also (P.7), which ensures the validity of \cref{rem:restriction_pos_comp_l1}.
A part of the argument can already be found in the proof~\cite{FPSS24}*{Th.~5.4}.
\end{remark}

\subsection{Boundedness of functional minimizers}\label{ssec:boundedness_CL}

We end this section with the following result, generalizing the classical $L^\infty$ bound for minimizers of~\eqref{eq:lambda_1}, on a non-empty, bounded, and open set $\Omega$, see~\cite{CC07}*{Th.~4}.

\begin{proposition}
\label{res:bounded_eigenf}
Let \ref{p:c} and \ref{p:ms} be in force.
If $\nupla u$ is a $(1,\Phi)$-eigen-$N$-function of $\Omega\subset \R^d$, then $\nupla u\in L^\infty(\Omega;\R^N)$, with
\begin{equation*}
\|\nupla u_i\|_{L^\infty}
\le 
\frac{1}{|B_1|}\left(\frac{H^{\Phi,N}(\Omega)}{\delta d}\right)^d
\quad
\text{for}\
i=1,\dots,N,
\end{equation*}
where $\delta>0$ is as in~\ref{p:c}.
\end{proposition}

\begin{proof}
Fix $j\in\set*{1,\dots,N}$. By the second part of \cref{res:H=Lambda}, we know that 
\begin{equation*}
\nupla E^t
=\big(\set*{\nupla u_1>t_1},\dots,\set*{\nupla u_j>t},\dots,\set*{\nupla u_N>t_N}\big)
\end{equation*}
is a $1$-adjusted $\Phi$-Cheeger $N$-cluster of $\Omega$ for a.e.\ $t_i\in[0,\|\nupla u_i\|_{L^\infty})$, for $i\in\set*{1,\dots,N}\setminus\set*{j}$, and for a.e.\ $t\in[0,\|\nupla u_j\|_{L^\infty})$. 
Hence, fixing any such $t_i\in[0,\|\nupla u_i\|_{L^\infty})$, for $i\in\set*{1,\dots,N}\setminus\set*{j}$,  by~\eqref{eq:prop_cheeger_meas} in \cref{res:props_cheeger}\ref{item:prop_cheeger_meas}, we can estimate
\begin{equation*}
|\set*{\nupla u_j>t}|
=
|\nupla E^t_j|
\ge 
|B_1| \left(\frac{\delta d}{H^{\Phi,N}(\Omega)}\right)^d
\quad 
\text{for a.e.}\ 
t\in[0,\|\nupla u_j\|_{L^\infty}),
\end{equation*}
where $\delta>0$ is as in \ref{p:c}.
Since $\|\nupla u_j\|_{L^1}=1$ the conclusion readily follows by integrating the above inequality and using Cavalieri's principle.
\end{proof}

\begin{remark}
\cref{res:bounded_eigenf} may be achieved in several other settings, in the spirit of~\cite{FPSS24}, at least enforcing properties (P.1), (P.2), (P.4), and (P.7) of~\cite{FPSS24}*{Sect.~2.1}, and by requiring a finer version of the \emph{isoperimetric property} (P.6) of~\cite{FPSS24}*{Sect.~2.1}, e.g., see~\cite{FPSS24}*{Prop.~7.2} in the context of metric-measure spaces and the discussion in~\cite{FPSS24}*{Sect.~7.3} for non-local perimeter functionals.
In fact, \cref{res:bounded_eigenf} was inspired by the corresponding results in the non-local framework, e.g., see~\cite{BLP14}*{Rem.~7.3} and~\cite{BS22}*{Cor.~3.11}. 
\end{remark}

\section{Relation with the spectral problem}
\label{sec:bobkov-parini}

Just as we defined the $1$-geometric and the $1$-functional eigenvalues $\lstort^{\Phi, N}_{1,1}(F)$ and $\Lambda^{\Phi, N}_{1,1}(F)$ as variational problems set on $BV_0(F)$, in this section we treat their counterparts defined on $W^{1,p}_0(F)$. In particular, we show that the $(\Phi,N)$-Cheeger constant $H^{\Phi,N}(\Omega)$ of a non-empty, bounded, and open set $\Omega$ can be recovered as their limits as $p\to 1^+$, under suitable assumptions on the reference function $\Phi$. To do so, we adapt~\cite{BP18}*{Sect.~5}, where the authors deal with $\Phi=\|\cdot\|_\infty$.

\subsection{\texorpdfstring{$W^{1,p}_0$}{Sobolev} space and lower bound on \texorpdfstring{$\lambda_{1,p}$}{the first p-eigenvalue}}

Throughout this section, we let $p\in(1,\infty)$.
We begin with the following definition, in the same spirit of \cref{def:BV_0}. 

\begin{definition}[$W^{1,p}_0$ space]
\label{def:W_0}
Given a set $F\subset\R^d$, we let 
\begin{equation}
\label{eq:W_0}
W^{1,p}_0(F)
=
\set*{u\in W^{1,p}(\R^d) : u=0\ \text{a.e.\ on}\ \R^d\setminus F}, 
\end{equation}
and we let $\nupla u\in W^{1,p}_0(F;\R^N)$ if $\nupla u_i\in W^{1,p}_0(F)$ for $i=1,\dots,N$.
\end{definition}

\begin{remark}\label{rem:Sobolev_emb}
As similarly observed in \cref{rem:BV_emb} for the $BV$ space introduced in~\eqref{eq:BV_0},
 $W^{1,p}_0(F)$ may not coincide with the space of $W^{1,p}$ functions on $F$ with null trace at the boundary, unless $\partial F$ is sufficiently regular. Nevertheless, the usual Sobolev embeddings hold on a bounded $F$, as $W^{1,p}_0(F)\subset W^{1,p}_0(B_R)$ with $R>0$ such that $F \Subset B_R$. 
\end{remark}

\begin{remark}
We remark that in~\cite{BP18}, which we are extending, the authors consider a slightly different notion of Sobolev space. Fixed a non-empty, bounded, and open set $\Omega$, they define the Sobolev space $\widetilde W^{1,p}_0(F)$ for $F\subset\Omega$ as
\begin{equation}
\label{eq:W_0_tilde}
\widetilde W^{1,p}_0(F)
=
\set*{u\in W^{1,p}_0(\Omega) : u=0\ \text{a.e.\ on}\ \Omega\setminus F}.
\end{equation}
We stress that, if $\Omega$ is chosen Lipschitz, then the spaces in~\eqref{eq:W_0} and in~\eqref{eq:W_0_tilde} coincide.
\end{remark}

In analogy with~\eqref{eq:lambda_1}, we can introduce the following definition.  

\begin{definition}[First $p$-eigenvalue]
The \emph{first $p$-eigenvalue} of a set $F\subset\R^d$ is 
\begin{equation}
\label{eq:lambda_1,p}
\lambda_{1,p}(F)
=
\inf\set*{\int_F|\nabla u|^p\de x : u\in W^{1,p}_0(F),\ \|u\|_{L^p}=1}
\in[0,\infty].
\end{equation}
\end{definition}

\begin{remark}[Non-negative competitors]
\label{rem:restriction_pos_comp_lp}
Similarly to \cref{rem:restriction_pos_comp_l1}, the competitors in problem~\eqref{eq:lambda_1,p} can be taken non-negative, thanks to the chain rule for Sobolev functions.
\end{remark}

\begin{remark}\label{rem:ex_eigenfunctions}
As similarly observed in \cref{rem:ex_eigenfunctions_1}, by its very definition, $W_0^{1,p}(F)\ne\set*{0}$ if and only if $\lambda_{1,p}(F)<\infty$. If, in addition, $F$ is bounded, then there exist eigenfunctions, that is, functions $u\in W_0^{1,p}(F)$ realizing the infimum in~\eqref{eq:W_0}. To see this, it is enough to take an infimizing sequence (of non-negative competitors without loss of generality, in virtue of \cref{rem:restriction_pos_comp_lp}), to use the compact embeddings (refer to \cref{rem:Sobolev_emb}), and to exploit the lower semicontinuity of the Sobolev seminorm. 
In particular, this holds true for any non-empty, bounded, and open set $\Omega$. Moreover, if  $\partial\Omega$ is sufficiently regular, then $\lambda_{1,p}(\Omega)$ is the usual first eigenvalue of the Dirichlet $p$-Laplacian on~$\Omega$. Nevertheless, we warn the reader that, if $\Omega$ is not regular, then $\lambda_{1,p}(\Omega)$ might be smaller than the usual first eigenvalue of the Dirichlet $p$-Laplacian, see, e.g.,~\cite{Lin93}.
\end{remark}

The following result rephrases~\cite{LW97}*{App.}, see also~\cite{KF03}*{Th.~3} and~\cite{FPSS24}*{Cor.~6.4}.
We provide a sketch of its proof for the reader's convenience. 

\begin{theorem}[Lower bound on $\lambda_{1,p}(F)$]
\label{res:cheeger_p}
Given a set $F\subset\R^d$, it holds
\begin{equation*}
\lambda_{1,p}(F)
\ge
\left(\frac{h(F)}{p}\right)^p.
\end{equation*}
\end{theorem}

\begin{proof}
Assuming $\lambda_{1,p}(F)<\infty$ without loss of generality, we can find $u\in W^{1,p}_0(F)$ with $\|u\|_{L^p}=1$.
A simple application of the chain rule yields that $v=|u|^{p-1}u\in BV_0(F)$ with $\|v\|_{L^1}=1$ and $|Dv|=p|u|^{p-1}|\nabla u|\mathscr L^d$.
Consequently, by \cref{res:h=lambda_1} and H\"older's inequality, we get that 
\begin{equation*}
h(F)
=
\lambda_{1,1}(F)
\le 
|Dv|(\R^d)
\le 
p\int_{\R^d}|u|^{p-1}|\nabla u|\de x
\le 
p\,\|\nabla u\|_{L^p},
\end{equation*}
so that $\|\nabla u\|_{L^p}^p\ge(h(F)/p)^p$, and the conclusion readily follows.
\end{proof}

\begin{remark}
The content of this subsection can be rephrased in the abstract setting of~\cite{FPSS24}, once a proper notion of Sobolev space is introduced.
We refer the reader to~\cite{FPSS24}*{Sects.~2.3.3 and 6.1}. 
We also stress that, in metric-measure spaces, one can rely on a plainer approach, see the discussion in~\cite{FPSS24}*{Sect.~7.1}.
\end{remark}

\subsection{First \texorpdfstring{$p$}{p}-geometric and \texorpdfstring{$p$}{p}-functional eigenvalues}

We introduce the following definition, in the spirit of the one given in~\cite{BP18}*{Sect.~5}, extending our \cref{def:lstort_1} to also cover the case $p>1$.

\begin{definition}[First $p$-geometric $(\Phi,N)$-eigenvalue]\label{def:lstort_p}
The \emph{first $p$-geometric $(\Phi,N)$-ei\-gen\-value} of a set $F\subset\R^d$ is 
\begin{equation}
\label{eq:lstort_p}
\lstort^{\Phi,N}_{1,p}(F)
=
\inf\set*{
\Phi(\lambda_{1,p}(\nupla E)) : \nupla E\ \text{is an $N$-set of $F$}
}\in[0,\infty],
\end{equation}
where, for brevity, we have set
\begin{equation*}
\lambda_{1,p}(\nupla E)
=
\big(\lambda_{1,p}(\nupla E_1),\dots,\lambda_{1,p}(\nupla E_N)\big).
\end{equation*} 
Any $N$-set $\E$ of $F$ achieving the infimum is a \emph{$(p,\Phi)$-eigen-$N$-set} of $F$. 
\end{definition}

Note that, as always for a non-empty, bounded, open set $\Omega$, we have $\lstort^{\Phi,N}_{1,p}(\Omega)<\infty$, as $\lambda_{1,p}(\nupla E_i)<\infty$ for $i=1,\dots,N$, simply by choosing $\nupla E$ as any collection of $N$ disjoint open balls contained in $\Omega$.

Just as we gave a functional counterpart to \cref{def:lstort_1} with \cref{def:Lambda_1}, we also define the functional counterpart to the previous \cref{def:lstort_p}. To do so, we first define our competitors, in analogy to \cref{def:1N-function} (see also the auxiliary problem introduced in the proof of~\cite{BP18}*{Prop.~5.1}).

\begin{definition}[$(p,N)$-function]
\label{def:pN_function}
We say that $\nupla u\in W^{1,p}_0(F;\R^N)$ is a \emph{$(p,N)$-function} of $F\subset\R^d$ if $\nupla u_i\ge0$, $\|\nupla u_i\|_{L^p}=1$ and $\nupla u_i\,\nupla u_j=0$ a.e.\ in $F$ whenever $i\neq j$, for $i,j=1,\dots,N$.
\end{definition}

We can now introduce the following definition, which extends \cref{def:Lambda_1} to $p>1$.

\begin{definition}[First $p$-functional $(\Phi,N)$-eigenvalue]
The \emph{first $p$-functional $(\Phi,N)$-ei\-gen\-value} of a set $F\subset\R^d$ is 
\begin{equation}
\label{eq:Lambda_p}
\Lambda^{\Phi,N}_{1,p}(F)
=
\inf\set*{
\Phi([\nupla u]_{p,F}^p) : \nupla u\ \text{is a $(p,N)$-function of $F$}
}\in[0,\infty],
\end{equation}
where, for brevity, we have set
\begin{equation*}
[\nupla u]_{p,F}^p
=
\big(\|\nabla\nupla u_1\|_{L^p}^p,\dots,\|\nabla\nupla u_N\|_{L^p}^p\big)
\end{equation*} 
and, if no confusion can arise, we shall drop the reference to the ambient set $F$ and write $[\nupla u]_{p}^p$.
Any $(p,N)$-function $\nupla u$ of $F$ achieving the infimum is a \emph{$(p,\Phi)$-eigen-$N$-function} of~$F$. 
\end{definition} 

As always, when considering a non-empty, bounded, and open set $\Omega$, we have $\Lambda^{\Phi,N}_{1,p}(\Omega)<\infty$, as a viable competitor is given by an $N$-tuple of Sobolev functions supported on $N$ disjoint open balls contained in $\Omega$.

The following result is the analog of \cref{res:existence_f}, see also the proof of~\cite{BP18}*{Prop.~5.1}, ensuring existence of minimizers when $\Omega$ is a non-empty, bounded, and open set.

\begin{theorem}[Existence of minimizers of $\Lambda^{\Phi,N}_{1,p}(\Omega)$]
\label{res:existence_p}
Let \ref{p:lsc}, \ref{p:c}, and \ref{p:m} be in force.
Then, $(p,\Phi)$-eigen-$N$-functions of $\Omega$ exist.
\end{theorem}

\begin{proof}
Let $\set*{\nupla u^k:k\in\N}$ be an infimizing sequence for $\Lambda_{1,p}^{\Phi,N}(\Omega)$ and let $\eps>0$.
By \ref{p:c}, for all $k\in\N$ sufficiently large we have that 
\begin{equation*}
\Lambda_{1,p}^{\Phi,N}(\Omega)+\eps
\ge 
\Phi([\nupla u^k]_{p}^p)
\ge 
\delta\|\nabla\nupla u^k_i\|_{L^p}^p,
\end{equation*}
where $\delta>0$ is as in \ref{p:c}.
Since $\Omega\subset\R^d$ is bounded, the embedding $W^{1,p}_0(\Omega)\subset L^p(\Omega)$ is compact.
Thus, up to subsequences, $\nupla u^k_i\to\nupla u_i$ as $k\to\infty$ in $L^p(\Omega)$ for $i=1,\dots,N$, for some $\nupla u_i\in L^p(\Omega)$.
It is easy to see that $\nupla u$ is a $(p,N)$-function of $\Omega$ with 
\begin{equation*}
\Phi([\nupla u]_{p}^p)
\le 
\Phi\left(\liminf_{k\to\infty} 
[\nupla u^k]_{p}^p\right)
\le 
\liminf_{k\to\infty} 
\Phi\left([\nupla u^k]_{p}^p\right)
=
\Lambda_{1,p}^{\Phi,N}(\Omega)
\end{equation*}
by the lower semicontinuity of the seminorms, \ref{p:m}, and \ref{p:lsc}. The claim follows.
\end{proof}

\begin{remark}[More general version of \cref{res:existence_p}]
Similarly to \cref{rem:existence_f}, to ensure \cref{res:existence_p} it is enough to assume that $\Omega\subset\R^d$ is a bounded measurable set with $|\Omega|>0$ containing at least a viable competitor.
As for \cref{res:existence_f}, the boundedness of $\Omega$ cannot be relaxed to $|\Omega|<\infty$, as this does not necessarily guarantee the compactness of the embedding $W^{1,p}_0(\Omega)\subset L^p(\Omega)$. 
For a more detailed discussion, see~\cite{Maz11book}*{Sect.~6.4.3} which contains the sharp hypotheses to ensure the compactness of the embedding. 
\end{remark}

\begin{remark}
The content of this subsection can be rephrased in the abstract setting of~\cite{FPSS24}, once suitable Sobolev spaces are available, see~\cite{FPSS24}*{Sects.~2.3.3, 6.1, and 7.1}.
\end{remark}

The following result states that the $p$-geometric and $p$-functional eigenvalues for a non-empty, bounded, and open set $\Omega$ coincide. 
\cref{res:equality_p} is the analog of \cref{res:lstort=Lambda_1}, and we omit its proof since it can be repeated almost \emph{verbatim}.

\begin{theorem}[$\Lambda_{1,p}^{\Phi,N}(\Omega)=\lstort_{1,p}^{\Phi,N}(\Omega)$]
\label{res:equality_p}
The following holds
\begin{equation*}
\lstort_{1,p}^{\Phi,N}(\Omega)\ge\Lambda_{1,p}^{\Phi,N}(\Omega).    
\end{equation*}
If \ref{p:m} is in force, then
\begin{equation*}
\lstort_{1,p}^{\Phi,N}(\Omega)=\Lambda_{1,p}^{\Phi,N}(\Omega).    
\end{equation*}
Moreover, if $\nupla u$ is a $(p,\Phi)$-eigen-$N$-function of $\Omega$, then 
\begin{equation}
\label{eq:function_to_set}
\nupla E
=
\big(\set*{\nupla u_1>0},\dots,\set*{\nupla u_N>0}\big)
\end{equation}
is a $(p,\Phi)$-eigen-$N$-set of $\Omega$.
Viceversa, if $\E$ is a $(p,\Phi)$-eigen-$N$-set of $\Omega$, there exists a $(p,\Phi)$-eigen-$N$-function $\nupla u$ such that $\set*{\nupla u_i>0} \subset \E_i$ for all $i=1,\dots,N$.
\end{theorem}

\begin{remark}
Analogously to the case $p=1$, \cref{res:equality_p} yields that, up to possibly passing to a smaller $N$-subset, each chamber of a $(p,\Phi)$-eigen-$N$-set of $\Omega$ is the zero superlevel set of a $(p,\Phi)$-eigen-$N$-function of $\Omega$.
Actually, if a chamber is open, then the set inclusion is an equality, since the corresponding eigenfunction is strictly positive on the entire chamber as a consequence of Harnack's inequality, refer for instance to~\cite{KL06}*{Sect.~2}.
\end{remark}

\begin{remark}[More general version of \cref{res:equality_p}]
As in \cref{rem:equality_1}, the equality $\Lambda_{1,p}^{\Phi,N}(\Omega)=\lstort^{\Phi,N}_{1,p}(\Omega)$ can be achieved under weaker  assumptions on $\Omega$---in fact, more generally, within the abstract setting of~\cite{FPSS24}, at least enforcing properties (RP.1), (RP.2), (RP.3), (RP.4), (RP.+), and (RP.L) of~\cite{FPSS24}*{Sect.~2.3}, and also property (P.7) of~\cite{FPSS24}*{Sect.~2.1}.
For an account on the strategy, we refer to~\cite{FPSS24}*{Sect.~6.1}  (recall also the plainer approach available in the metric-measure framework, see~\cite{FPSS24}*{Sect.~7.1}).
\end{remark}

\cref{res:existence_p,res:equality_p} immediately yield the following result.

\begin{corollary}[Existence of minimizers of $\lstort^{\Phi,N}_{1,p}(\Omega)$]
\label{cor:existence_geom_p}
Let \ref{p:lsc}, \ref{p:c}, and \ref{p:m} be in force.
Then, $(p,\Phi)$-eigen-$N$-sets of $\Omega$ exist.
\end{corollary}

\subsection{Boundedness of functional minimizers}

We now provide an analog of \cref{res:bounded_eigenf} for minimizers of the problem~\eqref{eq:Lambda_p}, see \cref{res:bounded_eigenf_p} below, in the spirit of~\cite{BLP14}*{Th.~3.3} (see also~\cite{FP14}*{Th.~3.2}). To this aim, we first need to introduce some terminology, as follows.

\begin{definition}[$C^1$ smoothness]
\label{def:C1}
We say that $\Phi\colon\ort\to[0,\infty)$ is of \emph{class $C^1$} if, for any $\nupla v\in\R^N$, there exist an open neighborhood $V\subset\R^N$ of $\nupla v$ in $\R^N$ and $\widetilde\Phi\in C^1(V)$ such that $\widetilde\Phi=\Phi$ on $V\cap\ort$.
In this case, we let $\nabla\Phi(\nupla v)=\nabla\widetilde\Phi(\nupla v)$.
\end{definition}

It is worth noticing that, if $\Phi$ is of class $C^1$ as in \cref{def:C1}, then $\nabla\Phi(\nupla v)$ depends neither on the choice of the neighborhood $V$ of $\nupla v$ in $\R^N$ nor of the extension $\widetilde\Phi$ of $\Phi$ in $V$, but only on the values of $\Phi$ in the closed cone $\ort$.
In particular, if $\Phi$ is of class $C^1$, then it is of class $C^1$ in the interior of $\ort$.
Furthermore, as the reader may observe, \cref{def:C1} may be relaxed in several ways, as it is not needed in its full force in the results below.
We prefer not to stress this point here, as it is not of crucial importance. 

We can now state the following analog of \cref{res:bounded_eigenf}.
Note that we do not treat the case $p>N$, as in this case the boundedness of minimizers of~\eqref{eq:Lambda_p} trivially follows from Morrey's inequality.

\begin{theorem}
\label{res:bounded_eigenf_p}
Let $\Phi$ of class $C^1$, $p\le N$, and $\nupla u$ be a $(p,\Phi)$-eigen-$N$-function of $\Omega\subset \R^d$.
If $\partial_i\Phi([\nupla u]_p^p)>0$ for some $i\in\set*{1,\dots,N}$, then $\nupla u_i\in L^\infty(\Omega)$, with
\begin{equation}
\label{eq:bounded_eigenf_p}
\|\nupla u_i\|_{L^\infty}
\le 
C_i,
\end{equation}
where $C_i>0$ depends on $d$, $p$, and $\lambda_{1,p}(\set*{u_i>0})$, and also on $\Omega$ if $p=N$, but is independent of $\Phi$.
\end{theorem}

For the proof of \cref{res:bounded_eigenf_p}, we need the following simple preliminary result.

\begin{lemma}
\label{res:pde_p}
Let $\Phi$ be of class $C^1$.
If $\nupla u$ is a $(p,\Phi)$-eigen-$N$-function of $\Omega$, then
\begin{equation}
\label{eq:pde_p}
\partial_i\Phi([\nupla u]_p^p)
\,
\left(
\int_{\Omega}|\nabla\nupla u_i|^{p-2}\,\scalar*{\nabla\nupla u_i,\nabla\varphi}\de x
-
\lambda_{1,p}(\set*{\nupla u_i>0})
\int_{\Omega}|\nupla u_i|^{p-2}\,\nupla u_i\,\varphi\de x
\right)
=0
\end{equation}
for every $\varphi\in W^{1,p}_0(\set*{\nupla u_i>0})$ and $i=1,\dots,N$.
\end{lemma}

\begin{proof}
Without loss of generality, we may assume $i=1$.
Let $\varphi\in W^{1,p}_0(\set*{\nupla u_1>0})$ be fixed.
For $\eps\in\R$, we define $\nupla u^\eps=(\nupla u_1^\eps,\nupla u_2,\dots,\nupla u_N)$, where 
\begin{equation*}
\nupla u_1^\eps
=
\frac{|\nupla u_1+\eps\varphi|}{\|\nupla u_1+\eps\varphi\|_{L^p}}.
\end{equation*}
By definition, $\nupla u^\eps$ is a $(p,N)$-function of $\Omega$, with $\nupla u^\eps|_{\eps=0}=\nupla u$.
Due to the minimality of~$\nupla u$, the map $\eps\mapsto\Phi([\nupla u^\eps]_p^p)$  
achieves a local minimum at $\eps=0$.
Our aim is now to compute the derivative of this map at $\eps=0$.
Let us start by observing that 
\begin{equation}
\label{eq:sigaro}
[\nupla u^\eps]_{p}^p
=
\left(
\frac{\|\nabla\nupla u_1+\eps\nabla\varphi\|_{L^p}^p}{\|\nupla u_1+\eps\varphi\|_{L^p}^p},
\|\nabla\nupla u_2\|_{L^p}^p,
\dots,
\|\nabla\nupla u_N\|_{L^p}^p
\right)
\end{equation} 
for $\eps\in\R$ since, by the chain rule, $|\nabla|\nupla u_1+\eps\varphi||=|\nabla\nupla u_1+\eps\nabla\varphi|$ a.e.\ in $\Omega$. 
We observe that $\eps\mapsto\|\nabla\nupla u_1+\eps\nabla\varphi\|_{L^p}^p
$ and $\eps\mapsto\|\nupla u_1+\eps\varphi\|_{L^p}^p$ are of class $C^1$ in a neighborhood of $\eps=0$, since $p>1$, the map $t\mapsto |t|^{p}$ belongs to $ C^1(\R)$, with derivative equal to $t\mapsto p\,|t|^{p-2}\,t\in C^0(\R)$.
Moreover, owing to H\"older's inequality,
\begin{equation*}
(|\nupla u_1|+c\,|\varphi|)^{p-1}\,|\varphi|\quad \text{and} \quad
(|\nabla\nupla u_1|+c\,|\nabla\varphi|)^{p-1}\,|\nabla\varphi|
\end{equation*}
are in $L^1(\Omega)$ whenever $c\ge0$. Consequently, by differentiating under the integral sign, we get that
\begin{equation*}
\frac{\de}{\de\eps}
\int_{\Omega}|\nupla u_1+\eps\varphi|^p\de x
=
p
\int_{\Omega}|\nupla u_1+\eps\varphi|^{p-2}\,(\nupla u_1+\eps\varphi)\,\varphi\de x
\end{equation*}  
and, similarly,
\begin{equation*}
\frac{\de}{\de\eps}
\int_{\Omega}|\nabla \nupla u_1+\eps\nabla \varphi|^p\de x
=
p
\int_{\Omega}|\nabla\nupla u_1+\eps\nabla\varphi|^{p-2}\,\scalar*{\nabla \nupla u_1+\eps\nabla \varphi,\nabla\varphi}\de x,
\end{equation*}
both derivatives being continuous with respect to $\eps\in\R$.
Recalling that $\|\nupla u_1\|_{L^p}=1$, we see that $\|\nupla u_1+\eps\varphi\|_{L^p}^p\ge\sfrac12$ in a neighborhood of $\eps=0$, so that, by the quotient rule, the function in~\eqref{eq:sigaro} is of class $C^1$ in a neighborhood of $\eps=0$, with 
\begin{equation*}
\frac{\de}{\de\eps}
\frac{\|\nabla\nupla u_1+\eps\nabla\varphi\|_{L^p}^p}{\|\nupla u_1+\eps\varphi\|_{L^p}^p}
\bigg|_{\eps=0}
=
p
\int_{\Omega}|\nabla \nupla u_1|^{p-2}\,\scalar*{\nabla \nupla u_1, \nabla\varphi}\de x
-
p
\|\nabla\nupla u_1\|_{L^p}^p
\int_{\Omega}|\nupla u_1|^{p-2}\, \nupla u_1\, \varphi\de x.
\end{equation*}
Owing to the minimality of $\nupla u$, the regularity of $\Phi$, the chain rule, and the validity of the equality $\|\nabla \nupla u_1\|_{L^p}^p=\lambda_{1,p}(\set*{\nupla u_1>0})$, we hence get that 
\begin{equation*}
\begin{split}
0&=\frac{\de}{\de\eps}\Phi([\nupla u^\eps]_p^p)\bigg|_{\eps=0}
=
\partial_1\Phi([\nupla u]_p^p)
\,
\frac{\de}{\de\eps}
\frac{\|\nabla\nupla u_1+\eps\nabla\varphi\|_{L^p}^p}{\|\nupla u_1+\eps\varphi\|_{L^p}^p}
\bigg|_{\eps=0}
\\
&=
p
\partial_1\Phi([\nupla u]_p^p)
\,
\left(
\int_{\Omega}|\nabla\nupla u_1|^{p-2}\,\scalar*{\nabla\nupla u_1,\nabla\varphi}\de x
-
\lambda_{1,p}(\set*{\nupla u_1>0})
\int_{\Omega}|\nupla u_1|^{p-2}\,\nupla u_1\,\varphi\de x
\right)
\end{split}
\end{equation*} 
yielding the conclusion.
\end{proof}

We are now ready to prove \cref{res:bounded_eigenf_p}.

\begin{proof}[Proof of \cref{res:bounded_eigenf_p}]
Without loss of generality, we may assume $i=1$.
We follow the same strategy of the proof of~\cite{BLP14}*{Th.~3.3}.

We deal with the case $p<N$.
To this aim, we let $M\in(0,\infty)$ and $\beta\ge1$, and we apply \cref{res:pde_p} with the choice $\varphi=(\min\set*{\nupla u_1,M})^\beta$. 
It is not difficult to infer that $\varphi\in W^{1,p}_0(\set*{\nupla u_1>0})$ as in \cref{def:W_0} thanks to the chain rule.
Since $\partial_i\Phi([\nupla u]_p^p)\ne0$ by assumption,  equality~\eqref{eq:pde_p} in \cref{res:pde_p} immediately yields that 
\begin{equation*}
\int_{\R^d}|\nabla\nupla u_1|^{p-2}\,\scalar*{\nabla\nupla u_1,\nabla\varphi}\de x
=
\lambda_{1,p}(\set*{\nupla u_1>0})
\int_{\R^d}|\nupla u_1|^{p-2}\,\nupla u_1\,\varphi\de x.
\end{equation*}
By definition of $\varphi$, we easily recognize that
\begin{equation*}
\int_{\R^d}|\nupla u_1|^{p-2}\,\nupla u_1\,\varphi\de x
\le 
\int_{\R^d}\nupla u_1^{p+\beta-1}\de x
\end{equation*}
and
\begin{equation*}
\begin{split}
\int_{\R^d}|\nabla\nupla u_1|^{p-2}\,\scalar*{\nabla\nupla u_1,\nabla\varphi}\de x
&=
\beta\int_{\set*{\nupla u_1 < M}}|\nabla\nupla u_1|^p\,\nupla u_1^{\beta-1}\de x
\\
&=
\frac{\beta\,p^p}{(p+\beta-1)^p}
\int_{\R^d}\left|\,\nabla\left(\min\set*{\nupla u_1,M}^{\frac{p+\beta-1}{p}}\right)\right|^{p}\de x.
\end{split}
\end{equation*}
Owing to the Sobolev inequality in $W^{1,p}(\R^d)$, we also infer that
\begin{equation*}
\int_{\R^d}\left|\,\nabla\left(\min\set*{\nupla u_1,M}^{\frac{p+\beta-1}{p}}\right)\right|^{p}\de x
\ge 
c_{d,p}^p
\left(\int_{\R^d}\left(\min\set*{\nupla u_1,M}^{\frac{p+\beta-1}{p}}\right)^{\frac{dp}{d-p}}\de x\right)^{\frac{d-p}{d}},
\end{equation*}
where $c_{d,p}>0$ is the Gagliardo--Nirenberg--Sobolev embedding constant, depending on $d$ and $p$ only.
By combining all the above inequalities and then passing to the limit as $M\to\infty$, we conclude that
\begin{equation*}
\left(\int_{\R^d}\left(\nupla u_1^{\frac{p+\beta-1}{p}}\right)^{\frac{dp}{d-p}}\de x\right)^{\frac{d-p}{d}}
\le 
\frac{\lambda_{1,p}(\set*{\nupla u_1>0})}{c_{d,p}^p}
\left(\frac{\beta+p-1}{p}\right)^{p-1}
\int_{\R^d}\left(\nupla u_1^{\frac{p+\beta-1}{p}}\right)^p\de x,
\end{equation*}
where we used that $\frac{\beta+p-1}{p}\frac{1}{\beta}\le1$, since $\beta\ge1$.
Inequality~\eqref{eq:bounded_eigenf_p} hence follows by the very same iteration argument used in the proof of~\cite{BLP14}*{Th.~3.3}.
In particular, note that the constant in~\eqref{eq:bounded_eigenf_p} depends neither on $\Omega$ nor on $\Phi$.

The borderline case $p=N$ follows similarly, as in the second part of the proof of~\cite{BLP14}*{Th.~3.3}.
Here we only observe that, since $\set*{\nupla u_1>0}\subset\Omega$ obviously, one can exploit the Sobolev inequality on $\Omega$, instead of that on $\R^d$.
Consequently, in this case, the constant in~\eqref{eq:bounded_eigenf_p} depends on $\Omega$ (but still not on $\Phi$).
\end{proof}

\begin{remark}
\label{rem:constant_boundedness}
As in~\cite{BLP14}*{Rem.~3.4}, a close inspection of the above proof of \cref{res:bounded_eigenf_p} yields that, for $p<N$, the constant $C_i>0$ in~\eqref{eq:bounded_eigenf_p} is given by
\begin{equation*}
C_i
=
\left(\frac{d}{d-p}\right)^{\frac{d(d-p)}{p^{\scalebox{0.3}{2}}}\frac{p-1}p}
\left(\frac{\lambda_{1,p}(\set*{\nupla u_i>0}}{c_{d,p}^p}\right)^{\frac{d}{p^{\scalebox{0.3}{2}}}},
\end{equation*}
where $c_{d,p}>0$ is the Gagliardo--Nirenberg--Sobolev embedding constant.
We stress that $c_{d,p}$ is stable in the limit as $p\to1^+$ and tends to the isoperimetric constant in $\R^d$.
\end{remark}

\begin{remark}
\cref{res:bounded_eigenf_p} (as well as \cref{res:pde_p}) may be achieved in more general settings, in the spirit of the general approach of~\cite{FPSS24} (see the examples detailed in~\cite{FPSS24}*{Sect.~7}).
In particular, \cref{res:bounded_eigenf_p} can be achieved in the fractional case (and naturally in several more general non-local frameworks, once suitable Sobolev-type embeddings are at disposal, see~\cites{BS22,FP14}), by naturally generalizing~\cite{BLP14}*{Th.~3.3} to the present setting. 
\end{remark}

\subsection{Limit of the spectral problem}

The main result of this section shows that the constant $H^{\Phi, N}(\Omega)$ can be recovered as the limit of $\lstort^{\Phi, N}_{1,p}(\Omega)$ as $p\to 1^+$, under suitable assumptions on the reference function $\Phi$ and (weak) regularity requests on the non-empty, bounded, open set $\Omega$, generalizing~\cite{BP18}*{Th.~5.3}.

\begin{theorem}[Limit of the spectral problem]
\label{res:limit}
Let \ref{p:lsc} and \ref{p:m} be in force.
Then,
\begin{equation}
\label{eq:liminf_p}
\liminf_{p\to1^+}
\lstort^{\Phi,N}_{1,p}(\Omega)
\ge 
H^{\Phi,N}(\Omega).
\end{equation}
In addition, enforcing \ref{p:+} and \ref{p:c}, if $\per (\Omega)<\infty$ and $\mathscr H^{d-1}(\partial\Omega\setminus\partial^*\Omega)=0$, then 
\begin{equation}
\label{eq:limsup_p}
\limsup_{p\to1^+}
\lstort^{\Phi,N}_{1,p}(\Omega)
\le 
H^{\Phi,N}(\Omega),
\end{equation}
so that, in this case, 
$H^{\Phi,N}(\Omega)=\lim\limits_{p\to1^+}\lstort^{\Phi,N}_{1,p}(\Omega)$.
\end{theorem}

\begin{proof}
We begin by proving~\eqref{eq:liminf_p}. 
Let $\eps>0$ and let $\nupla E$ be an $N$-set such that
\[ 
\eps+\lstort^{\Phi,N}_{1,p}(\Omega)
\ge
\Phi(\lambda_{1,p}(\nupla E)).
\]
Recalling the inequality of \cref{res:cheeger_p}, applying it to every chamber of $\nupla E$, and owing to~\ref{p:m}, we have
\[
\eps+\lstort^{\Phi,N}_{1,p}(\Omega)
\ge
\Phi\left(\left(\frac{h(\nupla E_1)}{p}\right)^p,\dots,\left(\frac{h(\nupla E_N)}{p}\right)^p\right).
\]
Now taking the $\liminf$ as $p\to 1^+$, owing to the lower semicontinuity~\ref{p:lsc}, to \cref{res:h=lambda_1}, and to \cref{res:H=lstort}, we get
\[
\eps+\liminf_{p\to 1^+} \lstort^{\Phi,N}_{1,p}(\Omega)
\ge
\Phi(\lambda_{1,1}(\nupla E))
\ge 
\lstort^{\Phi,N}_{1,1}(\Omega)
=
H^{\Phi,N}(\Omega).
\]
The claim now follows by letting $\eps\to 0$.

We now prove~\eqref{eq:limsup_p} assuming~\ref{p:c}, the stronger~\ref{p:+}, that $\per(\Omega)<\infty$, and that $\mathscr H^{d-1}(\partial\Omega\setminus\partial^*\Omega)=0$.
Fix any $1$-adjusted $\Phi$-Cheeger $N$-cluster $\E$ of $\Omega$ given by \cref{res:existence} paired with \cref{res:props_cheeger}\ref{item:prop_cheeger_modified_1-adj}. 
By \cref{res:regularity}\ref{item:reg_equiv}, we can assume that each $\nupla E_i$ is open.
Consequently, we can find $N$-clusters $\set*{\nupla E^k:k\in\N}$ of $\Omega$ as in \cref{res:approx} such that, thanks to \ref{p:+}, 
\begin{equation}
\label{eq:pane}
\Phi\left(\frac{\per(\nupla E^k)}{|\nupla E^k|}\right)
\le 
H^{\Phi,N}(\Omega)+\frac1k
\quad
\text{for}\ 
k\in\N,
\end{equation}
with $\nupla E^k_i \Subset \nupla E_i$ for $i=1,\dots, N$. Now, given $\eps>0$, we let 
\begin{equation*}
\nupla E^{k,\eps}_i
=
\set*{x\in\R^N : \operatorname{dist}(x,\nupla E_i^k)<\eps}
\quad
\text{for}\
i=1,\dots,N\
\text{and}\
k\in\N.
\end{equation*} 
Now fix $k\in\N$.
Possibly taking a smaller $\eps>0$ depending on the chosen~$k$, we have that $\nupla E^k_i \Subset \nupla E^{k,\eps}_i\Subset\nupla E_i$ for $i=1,\dots,N$.
Now let $\nupla v^{k,\eps}\in W^{1,p}(\Omega;\R^N)$ be such that  $\nupla v^{k,\eps}_i=1$ on $\nupla E^k_i$, $\nupla v^{k,\eps}_i=0$ on $\Omega\setminus\nupla E^{k,\eps}_i$ and $\nabla\nupla v^{k,\eps}_i=1/\eps$ on $\nupla E^{k,\eps}_i\setminus\nupla E^{k}_i$, for $i=1,\dots,N$.
Then, by construction, 
\begin{equation*}
\nupla u^{k,\eps}
=
\left(\frac{\nupla v^{k,\eps}_1}{\|\nupla v^{k,\eps}_1\|_{L^p}},\dots,\frac{\nupla v^{k,\eps}_N}{\|\nupla v^{k,\eps}_N\|_{L^p}}\right)
\end{equation*}
is a $(p,N)$-function of $\Omega$ as in \cref{def:pN_function}, with
\begin{equation}
\label{eq:vino}
[\nupla u^{k,\eps}]_{p}^p
=
\left(
\frac{\|\nabla\nupla v^{k,\eps}_1\|_{L^p}^p}{\|\nupla v^{k,\eps}_1\|_{L^p}^p},\dots,\frac{\|\nabla\nupla v^{k,\eps}_N\|_{L^p}^p}{\|\nupla v^{k,\eps}_N\|_{L^p}^p}\right)
\le
\left(
\frac{|\nupla E^{k,\eps}_1\setminus\nupla E^{k}_1|}{\eps^p\,|\nupla E^{k}_1|},
\dots,
\frac{|\nupla E^{k,\eps}_N\setminus\nupla E^{k}_N|}{\eps^p\,|\nupla E^{k}_N|}
\right).
\end{equation}
Since clearly $\lambda_{1,p}(\nupla E)\le[\nupla u^{k,\eps}]_{p}^p$, thanks to \ref{p:m} we can hence estimate 
\begin{equation}
\label{eq:acqua}
\lstort^{\Phi,N}_{1,p}(\Omega)
\le 
\Phi(\lambda_{1,p}(\nupla E))
\le 
\Phi([u^{k,\eps}]_{p}^p).
\end{equation}
Now, by well-known results (e.g., see~\cite{ACV08}*{Cor.~1}), we have that 
\begin{equation}
\label{eq:minkowski}
|\nupla E^{k,\eps}_i\setminus\nupla E^{k}_i|=\eps\,\per(\nupla E^{k}_i)+o(\eps)
\quad
\text{as}\
\eps\to0^+.
\end{equation}
Thus, by combining~\eqref{eq:vino} and~\eqref{eq:minkowski}, we get that 
\begin{equation*}
[\nupla u^{k,\eps}]_{p}^p
\le
\eps^{1-p}\left(
\frac{\per(\nupla E^{k}_1)}{|\nupla E^{k}_1|}+\frac{o(\eps)}{\eps},
\dots,
\frac{\per(\nupla E^{k}_N)}{|\nupla E^{k}_N|}+\frac{o(\eps)}{\eps}
\right)
\quad
\text{as}\ 
\eps\to0^+.
\end{equation*}
Exploiting~\ref{p:m} we first use the above inequality in~\eqref{eq:acqua}, and then the continuity~\ref{p:+}, to conclude that 
\begin{equation*}
\begin{split}
\limsup_{p\to1^+}
\lstort^{\Phi,N}_{1,p}(\Omega)
&\le 
\limsup_{p\to1^+}
\Phi
\left(
\eps^{1-p}\left(
\frac{\per(\nupla E^{k}_1)}{|\nupla E^{k}_1|}+\frac{o(\eps)}{\eps},
\dots,
\frac{\per(\nupla E^{k}_N)}{|\nupla E^{k}_N|}+\frac{o(\eps)}{\eps}
\right)
\right)
\\
&=
\Phi
\left(
\frac{\per(\nupla E^{k}_1)}{|\nupla E^{k}_1|}+\frac{o(\eps)}{\eps},
\dots,
\frac{\per(\nupla E^{k}_N)}{|\nupla E^{k}_N|}+\frac{o(\eps)}{\eps}
\right).
\end{split}
\end{equation*}
Once again exploiting \ref{p:+} and recalling~\eqref{eq:pane}, we pass to the limit as $\eps\to0^+$ to get
\begin{equation*}
\limsup_{p\to1^+}
\lstort^{\Phi,N}_{1,p}(\Omega)
\le 
\Phi
\left(
\frac{\per(\nupla E^{k})}{|\nupla E^{k}|}
\right)
\le 
H^{\Phi,N}(\Omega)+\frac1k
\quad
\text{for}\ 
k\in\N,
\end{equation*}
and now the claim follows by letting $k\to\infty$.
\end{proof}

\begin{remark}
The first part of Theorem~\ref{res:limit} may be achieved in more general contexts, following the line of~\cite{FPSS24}, by relying on the extension of Theorem~\ref{res:cheeger_p} in the abstract setting, see~\cite{FPSS24}*[Cor.~6.4].
\end{remark}

\subsection{Convergence of functional minimizers}

The following result proves that minimizers of~\eqref{eq:Lambda_p} converge to minimizers of~\eqref{eq:Lambda_1} as $p\to1^+$, under the same set of assumptions of \cref{res:limit}. This is in the same spirit of~\cite{BLP14}*{Th.~7.2}. 

\begin{theorem}
\label{res:conv_min}
Let \ref{p:lsc}, \ref{p:c}, and \ref{p:m} be in force.
Let $(p_k)_{k\in\N}\subset(1,\infty)$ be such that $p_k\to1^+$ as $k\to\infty$ and 
$\liminf_{k}
\Lambda^{\Phi,N}_{1,p_k}(\Omega)<\infty$.
If $\nupla u^k$ is a $(p_k,\Phi)$-eigen-$N$-function of $\Omega$ for each $k\in\N$, then there exists a $(1,N)$-function $\nupla u$ of $\Omega$ such that, up to passing to a subsequence, $\nupla u^k\to \nupla u$ in $L^1(\Omega;\R^N)$ as $k\to\infty$ and 
\begin{equation*}
\Phi([\nupla u]_1)
\le
\liminf_{k\to\infty}
\Lambda^{\Phi,N}_{1,p_k}(\Omega).
\end{equation*}
In addition, enforcing \ref{p:+}, if $\per(\Omega)<\infty$ and $\mathscr H^{d-1}(\partial\Omega\setminus\partial^*\Omega)=0$, then the limit $\nupla u$ is a $(1,\Phi)$-eigen-$N$-function of $\Omega$.
\end{theorem}

\begin{proof}
Since $\liminf_{k}
\Lambda^{\Phi,N}_{1,p_k}(\Omega)<\infty$, up to passing to a subsequence, without loss of generality we may assume that $C=\sup_{k}
\Lambda^{\Phi,N}_{1,p_k}(\Omega)<\infty$.
Since $\Omega\subset\R^d$ is bounded, we can find $R>0$ such that $\Omega\Subset B_R$. 
Since $\nupla u^k=0$ on $\R^d\setminus B_R$ according to \cref{def:pN_function}, by H\"older's inequality, \ref{p:c}, and \ref{p:m}, we can estimate
\begin{equation}
\label{eq:salciccia}
\|\nabla\nupla u^k_i\|_{L^1}
\le 
|B_R|^{1-\frac{1}{p_k}}\,\|\nabla \nupla u^k_i\|_{L^{p_k}}
\le
|B_R|^{1-\frac{1}{p_k}}\, \left (\frac{\Phi([\nupla u^k]^{p_k}_{p_k})}{\delta}\right)^{\frac{1}{p_k}}
\le
|B_R|^{1-\frac{1}{p_k}}
\left(\frac{C}{\delta}\right)^{\frac1{p_k}} 
\end{equation} 
for every $k\in\N$ and $i=1,\dots,N$, where $\delta>0$ is as in \ref{p:c}.
Since $p_k\to1^+$ as $k\to\infty$, the above inequality yields that $(\nupla u^k)_{k\in\N}$ is uniformly bounded in $BV_0(\Omega;\R^N)$. 
By the compactness of the embedding $BV_0(\Omega)\subset L^1(\Omega)$, we can find $\nupla u$ such that, up to subsequences, $\nupla u^k_i\to \nupla u_i$ in $L^1(\Omega)$ as $k\to\infty$ for $i=1,\dots,N$.
A plain argument proves that $\nupla u$ is a $(1,N)$-function of $\Omega$. 
By the lower semicontinuity of the $BV$ seminorm and the first inequality in~\eqref{eq:salciccia}, we have
\begin{equation}
\label{eq:risi}
[\nupla u]_1
\le 
\liminf_{k\to\infty}
[\nupla u^k]_1
\le 
\liminf_{k\to\infty}
|B_R|^{1-\frac1{p_k}}\,[\nupla u^k]_{p_k}
=
\liminf_{k\to\infty}
[\nupla u^k]_{p_k}
=
\liminf_{k\to\infty}
[\nupla u^k]_{p_k}^{p_k}
\end{equation}
and so, in virtue of \ref{p:m} and \ref{p:lsc}, we get that
\begin{equation}
\label{eq:bisi}
\Phi([\nupla u]_1)
\le 
\Phi\left(\liminf_{k\to\infty}
[\nupla u^k]_{p_k}^{p_k}\right) 
\le
\liminf_{k\to\infty}
\Phi([\nupla u^k]_{p_k}^{p_k})
=
\liminf_{k\to\infty}
\Lambda^{\Phi,N}_{1,p_k}(\Omega),
\end{equation} 
proving the first part of the statement.
The second part of the statement readily follows by combining the second part of  \cref{res:limit} with \cref{res:H=Lambda,res:equality_p}.
\end{proof}

\begin{remark}
Under the full set of assumptions of \cref{res:conv_min}, and additionally enforcing that $\Phi$ is of class $C^1$ and \ref{p:ms}, a simple interpolation argument allows to improve the $L^1$ convergence of minimizers in \cref{res:conv_min} to $L^q$ convergence for any $q\in[1,\infty)$ as in~\cite{BLP14}*{Th.~7.2}, thanks to  \cref{res:bounded_eigenf,res:bounded_eigenf_p}.
Indeed, given $q\in(1,\infty)$, for each $i=1,\dots,N$ we can estimate
\begin{equation*}
\|\nupla u^k_i-\nupla u_i\|_{L^q}
\le 
\|\nupla u^k_i-\nupla u_i\|_{L^1}^{1/q}
\,
\left(\|\nupla u^k_i\|_{L^\infty}+\|\nupla u_i\|_{L^\infty}\right).
\end{equation*}
Since \ref{p:ms} holds true, we must have $\partial_i\Phi(\nupla v)>0$ for any $\nupla v\in\ort\setminus\set*{0}$.
In virtue of \cref{res:bounded_eigenf_p,rem:constant_boundedness}, we hence just need to observe that, owing to \ref{p:c}, 
\begin{equation*}
\lambda_{1,p_k}(\set{\nupla u_i^k>0})
\le 
\frac{\Phi\left(\lambda_{1,p_k}(\set{\nupla u_i^k>0})\right)}{\delta}
\le 
\frac{\Phi([\nupla u^k]_p^p)}{\delta}
=
\frac{\Lambda^{\Phi,N}_{1,p_k}(\Omega)}{\delta}
\le 
\frac{C}{\delta}
\end{equation*}
for all $k\in\N$, where $\delta>0$ is as in \ref{p:c}.
Hence the constant appearing in~\eqref{eq:bounded_eigenf_p} is stable as $p_k\to1^+$, and thus $\sup_k\|\nupla u^k_i\|_{L^\infty}<\infty$, immediately yielding the conclusion.
\end{remark}

\begin{remark}
\cref{res:conv_min} may be achieved in more general settings, in the spirit of the general approach of~\cite{FPSS24} (see the examples detailed in~\cite{FPSS24}*{Sect.~7}).
In particular, \cref{res:conv_min} can be achieved in the fractional case (and in several more general non-local frameworks, once suitable Sobolev-type embeddings are at disposal, see~\cites{BS22,FP14}), by naturally generalizing~\cite{BLP14}*{Th.~7.2} to the present setting.
We nevertheless stress that, in the non-local framework, inequality~\eqref{eq:salciccia} has to be rephrased by using embeddings between non-local Sobolev spaces (e.g., see~\cite{BLP14}*{Lem.~2.6} in the fractional case), while the argument around~\eqref{eq:risi} and~\eqref{eq:bisi} should be replaced with an analogous one exploiting Fatou's Lemma (see the proof of~\cite{BLP14}*{Th.~7.2} for more details). 
\end{remark}

\subsection{Convergence of geometric minimizers}

The following result provides a geometric analog of \cref{res:conv_min}, proving that any $L^1$ limit of minimizers of~\eqref{eq:lstort_p} as $p\to1^+$ is a minimizer of~\eqref{eq:lstort_1}. 
In fact, having in mind the discussion around \cref{res:clusters_only}, we can prove something more, that is, any sequence of minimizers of~\eqref{eq:lstort_p} contains a sequence of $N$-clusters which is converging to a minimizer of~\eqref{eq:H} as $p\to1^+$. 

\begin{theorem}
\label{res:conv_geom}
Let \ref{p:lsc}, \ref{p:c}, and \ref{p:m} be in force.
Let $(p_k)_{k\in\N}\subset(1,\infty)$ be such that $p_k\to1^+$ as $k\to\infty$ and 
$\liminf_{k}
\lstort^{\Phi,N}_{1,p_k}(\Omega)<\infty$.
If $\nupla E^k$ is a $(p_k,\Phi)$-eigen-$N$-set of $\Omega$ for each $k\in\N$, then, up to subsequences, there exist $N$-clusters $\nupla F^k$ and $\nupla F$ of $\Omega$ such that 
\begin{equation*}
\nupla F^k_i\subset\nupla E^k_i
\quad
\text{and}
\quad
\nupla F^k_i\to\nupla F_i\
\text{in}\ L^1(\Omega),
\quad
\text{for each}\
i=1,\dots,N,
\end{equation*}
and, moreover,
\begin{equation*}
\Phi(\lambda_{1,1}(\nupla F))
\le
\liminf_{k\to\infty}
\lstort^{\Phi,N}_{1,p_k}(\Omega).
\end{equation*}
In addition, enforcing \ref{p:+}, if $\per(\Omega)<\infty$ and $\mathscr H^{d-1}(\partial\Omega\setminus\partial^*\Omega)=0$, then the limit $\nupla F$ is a $(1,\Phi)$-eigen-$N$-cluster of $\Omega$.
Moreover, under these assumptions, if $\nupla E^k_i\to\nupla E_i$ as $k\to\infty$ in $L^1(\Omega)$ for some $\nupla E_i\subset\Omega$, then $\nupla E=(\nupla E_1,\dots,\nupla E_N)$ is a $(1,\Phi)$-eigen-$N$-set of $\Omega$.
\end{theorem}

\begin{proof}
Since $\liminf_{k}
\lstort^{\Phi,N}_{1,p_k}(\Omega)<\infty$, up to passing to a subsequence, we may assume that $C=\sup_{k}
\lstort^{\Phi,N}_{1,p_k}(\Omega)<\infty$ without loss of generality.
Owing to \cref{res:h=lambda_1}, \cref{res:cheeger_p}, \ref{p:c}, and \ref{p:m}, we can hence bound
\begin{equation}
\label{eq:pastrano}
h(\nupla E^k_i)
=
\lambda_{1,1}(\nupla E^k_i)
\le
p_k\left(\frac{\Phi(\lambda_{1,p_k}(\nupla E^k))}{\delta}\right)^{\frac1{p_k}}
\le 
p_k\left(\frac{C}{\delta}\right)^{\frac1{p_k}}
\quad
\text{for}\ i=1,\dots,N.
\end{equation}
Being $h(\nupla E^k_i)<\infty$ and $\E^k_i$ bounded, each $\nupla E^k_i$ admits a Cheeger set $\nupla F^k_i\subset\nupla E^k_i$, see~\cite{FPSS24}*{Sect.~3.1}, so that $h(\nupla F^k_i)=\per(\nupla F^k_i)|\nupla F^k_i|^{-1}=h(\nupla E^k_i)$ for $i=1,\dots,N$.
Therefore $\nupla F^k=(\nupla F^k_1,\dots,\nupla F^k_N)$ is an $N$-cluster of $\Omega$ such that $h(\nupla F^k)=\per(\nupla F^k)|\nupla F^k|^{-1}=h(\nupla E^k)$ for each $k\in\N$. 
We now observe that, in virtue of the above inequality, 
\begin{equation*}
\per(\nupla F^k_i)
\le 
h(\nupla E^k_i)\,|\nupla F^k_i|
\le 
p_k\left(\frac{C}{\delta}\right)^{\frac1{p_k}}\,|\Omega|.
\end{equation*}
Owing to the equality $h(\nupla F^k_i) = h(\nupla E^k_i)$, the inequality~\eqref{eq:pastrano}, and the lower bound in~\cite{Leo15}*{Prop.~3.5(v)} to the measure of Cheeger sets, we have
\begin{equation*}
|\nupla F^k_i|
\ge 
|B_1| \left(\frac{d}{p_k}\,\left(\frac{\delta}{C}\right)^{\frac1{p_k}}\right)^{d},
\end{equation*}
for $k\in\N$ and $i=1,\dots,N$. 
Therefore, as $p_k\to 1^+$, up to passing to a subsequence, $\nupla F^k_i\to\nupla F_i$ as $k\to\infty$ in $L^1(\Omega)$ for each $i=1,\dots,N$, for some $\nupla F_i\subset\Omega$ with $|\nupla F_i|>0$. 
It is easy to see that $\nupla F=(\nupla F_1,\dots,\nupla F_N)$ is an $N$-cluster of $\Omega$ such that, owing to \cref{res:h=lambda_1}, the lower semicontinuity of the perimeter, and the equality $h(\nupla F^k_i) = h(\nupla E^k_i)$,
\begin{equation*}
\lambda_{1,1}(\nupla F)
=
h(\nupla F)
\le 
\frac{\per(\nupla F)}{|\nupla F|}
\le 
\liminf_{k\to\infty}
\frac{\per(\nupla F^k)}{|\nupla F^k|}
=
\liminf_{k\to\infty}
h(\nupla F^k)
=
\liminf_{k\to\infty}
h(\nupla E^k)
.
\end{equation*}
Combining the previous inequality with \cref{res:cheeger_p} gives
\begin{equation*}
\lambda_{1,1}(\nupla F)
\le 
\liminf_{k\to\infty}
h(\nupla E^k)
\le 
\liminf_{k\to\infty}
p_k\lambda_{1,p_k}(\nupla E^k)^{\frac1{p_k}}
=
\liminf_{k\to\infty}
\lambda_{1,p_k}(\nupla E^k).
\end{equation*}
Therefore, owing to \ref{p:m} and \ref{p:lsc}, we conclude that 
\begin{equation}
\label{eq:concorde}
\Phi(\lambda_{1,1}(\nupla F))
\le 
\Phi\left(\liminf_{k\to\infty}\lambda_{1,p_k}(\nupla E^k)\right)
\le
\liminf_{k\to\infty}\Phi(\lambda_{1,p_k}(\nupla E^k))
\le
\liminf_{k\to\infty}
\lstort^{\Phi,N}_{1,p_k}(\Omega),
\end{equation}
proving the first part of the statement.
The second part of the statement follows by combining~\eqref{eq:concorde} with the second part of \cref{res:limit} and \cref{res:H=lstort}. Moreover, if $\nupla E$ is as in the statement, then $\nupla F_i\subset\nupla E_i$ for each $i=1,\dots,N$, and so $\lambda_{1,1}(\nupla E)\le\lambda_{1,1}(\nupla F)$, yielding the minimality of $\nupla E$ and concluding the proof.
\end{proof}

\begin{remark}
\cref{res:conv_geom} may be achieved in more general settings, in the spirit of the general approach of~\cite{FPSS24}, as soon as suitable notions of isoperimetric inequality and Sobolev spaces are at disposal.
We refer the reader to~\cite{FPSS24}*{Sects.~2.3.3 and~6.1}.
We also stress that, in metric-measure spaces, one can rely on a plainer approach, see the discussion in~\cite{FPSS24}*{Sect.~7.1}. 
\end{remark}

\section{Stability with respect to the reference function}
\label{sec:stability}

In this section, we study the stability of the constants $H^{\Phi,N}(\Omega)$, $\Lambda^{\Phi,N}_{1,p}(\Omega)$, $\lstort^{\Phi,N}_{1,p}(\Omega)$, and of their corresponding minimizers with respect to the reference function $\Phi$. 
Throughout this section, we let $\Phi_k,\Phi\colon\ort\to[0,\infty)$, with $k\in\N$, be given reference functions. The following results hold for a non-empty, bounded, and open set $\Omega\subset \R^d$.

\subsection{Convergence of the constants}

We begin with the following simple result, dealing with the limit superior. We remark that \textit{no assumptions} at all are needed on each of the reference functions.

\begin{lemma}[Limsup]
\label{res:stability_limsup}
Let $\set*{\Phi_k:k\in\N}$, $\Phi$ be reference functions. If 
\begin{equation}
\label{eq:limsup_Phi_k}
\limsup\limits_{k\to\infty}\Phi_k\le\Phi,
\end{equation}
then the following hold:
\begin{enumerate}[label=(\roman*)]

\item 
\label{item:limsup_k_H}
$\limsup\limits_{k\to\infty}H^{\Phi_k,N}(\Omega)\le H^{\Phi,N}(\Omega)$;

\item\label{item:limsup_k_Lambda}
$\limsup\limits_{k\to\infty}\Lambda^{\Phi_k,N}_{1,p}(\Omega)\le\Lambda^{\Phi,N}_{1,p}(\Omega)$ for all $p\in[1,\infty)$;

\item\label{item:limsup_k_lstort}
$\limsup\limits_{k\to\infty}\lstort^{\Phi_k,N}_{1,p}(\Omega)\le\lstort^{\Phi,N}_{1,p}(\Omega)$ for all $p\in[1,\infty)$.

\end{enumerate}

\end{lemma}

\begin{proof}
We prove~\ref{item:limsup_k_H} only, the proof of~\ref{item:limsup_k_lstort} being identical, and that of~\ref{item:limsup_k_Lambda} very similar, just relying on $(p,N)$-functions rather than on $N$-sets.

Given any $N$-cluster $\nupla E$ of $\Omega$, we can estimate
\begin{equation*}
\limsup_{k\to\infty} 
H^{\Phi_k,N}(\Omega)
\le 
\limsup_{k\to\infty} 
\Phi_k\left({\frac{\per(\nupla E)}{|\nupla E|}}\right)
\le
\Phi\left({\frac{\per(\nupla E)}{|\nupla E|}}\right),
\end{equation*}
and the conclusion follows by passing to the infimum on the right-hand side.
\end{proof}

To deal with the limit inferior, we need to introduce the following definition.

\begin{definition}[Equicoercive sequence]
\label{def:equicoercive}
We say that the sequence $\set*{\Phi_k:k\in\N}$ is \emph{equicoercive} if each $\Phi_k$ satisfies \ref{p:c} with the same $\delta$, or, equivalently, if each $\Phi_k$ satisfies \ref{p:c} with $\delta_k$ with $\inf_k \delta_k >0$.
\end{definition}

We can now state the following result, dealing with the limit inferior. 
Here and in the following, the prefix $\Gamma$ in the (possibly, inferior or superior) limits denotes the usual notion of \emph{Gamma convergence} with respect to the Euclidean distance in $\ort$.
For an account, we refer the reader for instance to~\cite{Bra02book}.

\begin{proposition}[Liminf]
\label{res:stability_liminf}
Let $\set*{\Phi_k:k\in\N}$ be equicoercive and let $\Phi$ satisfy \ref{p:m}.
If 
\begin{equation}
\label{eq:G-liminf_Phi_k}
\Phi\le\Gamma\text{--}\liminf\limits_{k\to\infty}\Phi_k,
\end{equation}
then the following hold:
\begin{enumerate}[label=(\roman*)]

\item 
\label{item:liminf_k_H}
$H^{\Phi,N}(\Omega)\le\liminf\limits_{k\to\infty} H^{\Phi_k,N}(\Omega)$;

\item\label{item:liminf_k_Lambda}
$\Lambda^{\Phi,N}_{1,p}(\Omega)\le\liminf\limits_{k\to\infty}\Lambda^{\Phi_k,N}_{1,p}(\Omega)$ for all $p\in[1,\infty)$;

\item\label{item:liminf_k_lstort}
$\lstort^{\Phi,N}_{1,p}(\Omega)\le\liminf\limits_{k\to\infty}\lstort^{\Phi_k,N}_{1,p}(\Omega)$ for all $p\in[1,\infty)$.

\end{enumerate}

\end{proposition}

\begin{proof}
We only prove~\ref{item:liminf_k_H} and~\ref{item:liminf_k_Lambda}, as point~\ref{item:liminf_k_lstort} follows from~\ref{item:liminf_k_Lambda} and both parts of the statements of \cref{res:lstort=Lambda_1} (case $p=1$) and of \cref{res:equality_p} (case $p>1$), also owing to the hypothesis that the limit reference function satisfies~\ref{p:m}.

\vspace{1ex}

\textit{Proof of~\ref{item:limsup_k_H}}.
First, let us notice that, by the equicoercivity assumption and by the boundedness of $\Omega$, in virtue of \cref{res:lower_bound_H}, we have
\[
\liminf\limits_{k\to\infty} H^{\Phi_k, N}(\Omega) \ge N\delta d\left(\frac{|B_1|}{|\Omega|}\right)^{\frac 1d} > 0.
\]
Up to subsequences, we may thus assume that $\lim_{k}H^{\Phi_k,N}(\Omega)= C\in(0,\infty)$. 

Let us fix $\eps>0$ and, for all $k=1,\dots,N$, let $\nupla E^k$ be an $N$-cluster such that
\begin{equation}
\label{eq:pistrullo}
\eps+H^{\Phi_k,N}(\Omega) \ge \Phi_k\left(\frac{\per(\nupla E^k)}{|\nupla E^k|}\right).
\end{equation}
Owing to the fact that $\set*{\Phi_k:k\in\N}$ is equicoercive, we easily see that 
\begin{equation*}
\per(\nupla E^k_i)
\le 
\frac{2C}{\delta}\,|\nupla E^k_i|
\le 
\frac{2C}{\delta}\,|\Omega|
\quad
\text{for}\ i=1,\dots,N,
\end{equation*}
for all $k\in\N$ sufficiently large. 
Therefore, up to subsequences, $\nupla E^k_i\to\nupla E_i$ as $k\to\infty$ in $L^1(\Omega)$ for each $i=1,\dots,N$. With the same reasoning of \cref{res:props_cheeger}\ref{item:prop_cheeger_meas} via~\eqref{eq:pistrullo}, we get that
\[
|\E^k_i| 
\ge 
|B_1| \left(\frac{\delta d}{H^{\Phi_k, N}(\Omega)+\eps} \right)^d
\ge
|B_1| \left(\frac{\delta d}{2C}\right)^d,
\] 
thus showing that $\nupla E$ is an $N$-cluster of $\Omega$ with 
\begin{equation*}
{\frac{\per(\nupla E)}{|\nupla E|}}
\le 
\liminf_{k\to\infty}\frac{\per(\nupla E^k)}{|\nupla E^k|}.
\end{equation*}
Up to extracting a further subsequence, we may assume that $\lim_{k}\per(\nupla E^k)|\nupla E^k|^{-1}=\nupla v\in\ort$.
Owing to the choice of the subsequence, the assumption~\eqref{eq:G-liminf_Phi_k}, again the choice of the subsequence, and the assumption that $\Phi$ satisfies \ref{p:m}, we get that 
\begin{equation*}
\eps + \lim_{k\to\infty}H^{\Phi_k,N}(\Omega)
\ge
\lim_{k\to\infty}
\Phi_k\left(
\frac{\per(\nupla E^k)}{|\nupla E^k|}
\right)
\ge 
\Phi(\nupla v)
\ge 
\Phi\left(
{\frac{\per(\nupla E)}{|\nupla E|}}
\right)
\ge 
H^{\Phi,N}(\Omega).
\end{equation*}
Letting $\eps\to 0$, the validity of~\ref{item:liminf_k_H} follows.

\vspace{1ex}

\textit{Proof of~\ref{item:liminf_k_Lambda}}.
The argument is the same we used to prove~\ref{item:liminf_k_H}, so we only sketch it.

Fix $p\in[1,\infty)$ and, up to subsequences, assume that $\lim_{k}\Lambda^{\Phi_k,N}_{1,p}(\Omega)=C_p\in[0,\infty)$. 
Given any $\eps>0$, we can find a $(p,N)$-function $\nupla u^k$ of $\Omega$ such that
\[
\eps + \Lambda^{\Phi_k,N}_{1,p}(\Omega) \ge \Phi_k([\nupla u^k]^p_p)
\]
for each $k\in\N$ and such that 
\begin{equation*}
\|\nabla\nupla u^k_i\|_{L^p}^p
\le 
\frac{2C_p}{\delta},
\quad
\text{for}\ i=1,\dots,N,
\end{equation*}
for all $k\in\N$ sufficiently large.
By compactness of the embedding $W^{1,p}_0(\Omega)\subset L^p(\Omega)$ for $p>1$, or of the embedding $BV_0(\Omega)\subset L^1(\Omega)$ for $p=1$, up to subsequences, $\nupla u^k_i\to\nupla u$ as $k\to\infty$ in $L^p(\Omega)$ for $i=1,\dots,N$, for some $\nupla u_i\in L^p(\Omega)$.
It is easy to see that $\nupla u$ is a $(p,N)$-function of $\Omega$ with
\begin{equation*}
[\nupla u]_{p}^p
\le 
\liminf_{k\to\infty}
[\nupla u^k]_{p}^p.
\end{equation*} 
Again up to subsequences, we may assume that $\lim_{k}
[\nupla u^k]_{p}^p=\nupla v_p\in\ort$.
Just as before, owing to the choice of the subsequence, the assumption~\eqref{eq:G-liminf_Phi_k}, again the choice of the subsequence, and the assumption that $\Phi$ satisfies \ref{p:m}, we get that 
\begin{equation*}
\eps+\lim_{k\to\infty}\Lambda^{\Phi_k,N}_{1,p}(\Omega)
\ge
\lim_{k\to\infty}
\Phi_k([\nupla u^k]_{p}^p)
\ge 
\Phi(\nupla v_p)
\ge 
\Phi([\nupla u]_{p}^p)
\ge 
\Lambda^{\Phi,N}_{1,p}(\Omega),
\end{equation*}
and~\ref{item:liminf_k_Lambda} follows by taking the limit as $\eps\to 0$.
\end{proof}

As a consequence of \cref{res:stability_limsup,res:stability_liminf}, we get the following stability result. 
It is easy to observe that the combination of~\eqref{eq:limsup_Phi_k} and~\eqref{eq:G-liminf_Phi_k} yields~\eqref{eq:Gamma=pointwise}. 
We also point out that, assuming the reference functions $\Phi_k$ to satisfy \ref{p:m}, the validity of~\eqref{eq:Gamma=pointwise} implies that~\ref{p:m} holds for the limit reference function $\Phi$.

\begin{theorem}[Stability]
\label{res:stability}
Let $\set*{\Phi_k:k\in\N}$ be equicoercive. 
If
\begin{equation}
\label{eq:Gamma=pointwise}
\Phi=
\lim\limits_{k\to\infty}\Phi_k
=
\Gamma\text{--}\lim\limits_{k\to\infty}\Phi_k,
\end{equation} 
with $\Phi$ satisfying \ref{p:m}, then the following hold:
\begin{enumerate}[label=(\roman*)]

\item
$H^{\Phi,N}(\Omega)=\lim\limits_{k\to\infty} H^{\Phi_k,N}(\Omega)$;

\item
$\Lambda^{\Phi,N}_{1,p}(\Omega)
=
\lim\limits_{k\to\infty}\Lambda^{\Phi_k,N}_{1,p}(\Omega)$ for all $p\in[1,\infty)$;

\item
$\lstort^{\Phi,N}_{1,p}(\Omega)
=
\lim\limits_{k\to\infty}\lstort^{\Phi_k,N}_{1,p}(\Omega)$ for all $p\in[1,\infty)$.

\end{enumerate}
\end{theorem}

\begin{remark}[Application to $q$-norms]
The previous result applies to the family 
\begin{equation*}
\set*{\Phi_q=\|\cdot\|_q:q\in[1,\infty]}
\end{equation*}
as in~\eqref{eq:p_norm}, allowing to interpolate the results of~\cites{CL19,C17}, corresponding to $q=1$, with the ones of~\cites{BP18,P10}, corresponding to $q=\infty$. 
\end{remark}

\subsection{Convergence of minimizers}

We end the section with the convergence of minimizers with respect to the convergence of the reference functions, proving the counterparts of \cref{res:conv_min,res:conv_geom} in this situation.

The following result yields convergence of minimizers of~\eqref{eq:H} with respect to the convergence of the reference functions.
The proof is almost identical to that of \cref{res:stability_liminf}\ref{item:liminf_k_H} up to minor modifications, and so we omit it. 

\begin{theorem}
\label{res:conv1_stability}
Let the assumptions of \cref{res:stability_liminf} be in force, and let assume that $\liminf_{k}H^{\Phi_k,N}(\Omega)<\infty$, and let $\nupla E^k$ be a $\Phi_k$-Cheeger $N$-cluster of $\Omega$ for each $k\in\N$.
Up to subsequences, there exists an $N$-cluster $\nupla E$ of $\Omega$ such that 
\begin{equation*}
\nupla E^k_i\to\nupla E_i\ 
\text{in}\ L^1(\Omega),
\quad
\text{for each}\
i=1,\dots,N,
\end{equation*}
and
\begin{equation*}
\Phi\left(\frac{\per(\E)}{|\E|} \right)
\le 
\liminf_{k\to\infty}
H^{\Phi_k,N}(\Omega).
\end{equation*}
Moreover, under the assumptions of \cref{res:stability}, $\nupla E$ is a $\Phi$-Cheeger $N$-cluster of $\Omega$.
\end{theorem}

The following result yields convergence of minimizers of~\eqref{eq:Lambda_1} and~\eqref{eq:Lambda_p} with respect to the convergence of reference functions.
The proof is almost identical to that of \cref{res:stability_liminf}\ref{item:liminf_k_Lambda} up to minor modifications, and so we omit it.

\begin{theorem}
\label{res:conv2_stability}
Let $p\in[1,\infty)$,  let the assumptions of \cref{res:stability_liminf} be in force and assume that $\liminf_k \Lambda^{\Phi_k,N}_{1,p}(\Omega)<\infty$, and let $\nupla u^k$ be a $(p,\Phi_k)$-eigen-$N$-function of $\Omega$ for each $k\in\N$.
Up to subsequences, there exists a $(p,N)$-function $\nupla u$ of $\Omega$ such that 
\begin{equation*}
\nupla u^k_i\to\nupla u_i\
\text{in}\ L^p(\Omega),
\quad
\text{for}\ i=1,\dots,N,
\end{equation*}
and
\begin{equation*}
\Phi([\nupla u]_p^p)
\le 
\liminf_{k\to\infty}
\Lambda^{\Phi_k,N}_{1,p}(\Omega).
\end{equation*}
Moreover, under the assumptions of \cref{res:stability}, $\nupla u$ is a $(p,\Phi)$-eigen-$N$-function of $\Omega$.
\end{theorem}

Finally, similarly to \cref{res:conv_geom}, the following result proves that $L^1$ limits of minimizers of~\eqref{eq:Lambda_p} for a sequence of reference functions are minimizers of~\eqref{eq:Lambda_p} for the limit reference function. 

\begin{theorem}
\label{res:conv3_stability}
Let $p\in[1,\infty)$ and let the assumptions of \cref{res:stability_liminf} be in force, assume that $\liminf_k \lstort^{\Phi_k, N}_{1,p}(\Omega)<\infty$, and let $\nupla E^k$ be a $(p,\Phi_k)$-eigen-$N$-set of $\Omega$ for each $k\in\N$.
If $\nupla E^k\to\nupla E$ as $k\to\infty$ in $L^1(\Omega; \R^N)$ for some $N$-set $\E$ of $\Omega$, then
\begin{equation*}
\Phi(\lambda_{1,p}(\nupla E))
\le 
\liminf_{k\to\infty}
\lstort^{\Phi_k,N}_{1,p}(\Omega).
\end{equation*}
Moreover, under the assumptions of \cref{res:stability}, $\nupla E$ is a $(p,\Phi)$-eigen-$N$-set of $\Omega$.
\end{theorem}

\begin{proof}
We detail the proof in the case $p>1$ only, as the case $p=1$ is essentially the same but replacing  $W^{1,p}_0(\Omega;\R^N)$ with $BV_0(\Omega;\R^N)$. 

Owing to the fact that $\set*{\Phi_k:k\in\N}$ is equicoercive, the sequence $\{\,\lambda_{1,p}(\mathcal E^k)\,:\,k\in\mathbb{N}\,\}$ is bounded in $\ort$. Moreover, the chambers $\nupla E^k\subset\Omega$ are bounded.
Thus, by \cref{res:equality_p}, we find a $(p,\Phi_k)$-eigen-$N$-function $\nupla u^k$ of $\Omega$, i.e., such that 
\begin{equation*}
\|\nabla \nupla u^k_i\|_{L^p}^p=\lambda_{1,p}(\nupla E^k_i),
\quad
\|\nupla u_i^k\|_{L^p}=1,
\quad
\nupla u_i^k\ge0
\quad 
\text{and}
\quad
\nupla F_i^k = \set*{\nupla u^k_i >0} \subset \nupla E^k_i,
\end{equation*}
for each $k\in\N$ and $i=1,\dots,N$.
Consequently, the sequence $(\nupla u^k)_{k\in\N}$ of $(p,\Phi_k)$-eigen-$N$-functions of $\Omega$ is bounded in $W^{1,p}_0(\Omega;\R^N)$ and, thus, up to subsequences, $\nupla u^k\to\nupla u$ as $k\to\infty$ in $L^p(\Omega;\R^N)$ for some $(p,N)$-function $\nupla u$ of $\Omega$ such that $\|\nupla u_i\|_{L^p}=1$ for $i=1\,\dots,N$. 
By lower semicontinuity of the seminorm, we also infer that 
\begin{equation*}
[\nupla u]_p^p
\le 
\liminf_{k\to\infty}[\nupla u^k]_p^p
\le
\liminf_{k\to\infty}\lambda_{1,p}(\nupla E^k).
\end{equation*}
Again up to subsequences, we may assume that $\lim_{k}\lambda_{1,p}(\nupla E^k)=\nupla v_p\in\ort$.
Owing to \ref{p:m} and \cref{res:stability_liminf}, we thus get that 
\begin{equation}
\label{eq:cerbiatto}
\Phi(\lambda_{1,p}(\nupla F))
\le
\Phi([\nupla u]_p^p)
\le 
\Phi(\nupla v_p)
\le 
\liminf_{k\to\infty}
\Phi_k(\lambda_{1,p}(\nupla E^k))
= 
\liminf_{k\to\infty}
\lstort^{\Phi_k,N}_{1,p}(\Omega),
\end{equation}
where $\nupla F=\left(\set*{u_1>0},\dots,\set*{u_N>0}\right)$. Since $\nupla u^k\to\nupla u$ as $k\to\infty$ in $L^p(\Omega;\R^N)$ and $|\Omega|<\infty$, by Cavalieri's principle we infer that 
\begin{equation*}
\set*{\nupla u^k_i>t}\to\set*{\nupla u_i>t}
\quad
\text{as $k\to\infty$ in $L^1(\Omega)$ for a.e.\ $t\ge 0$.}
\end{equation*}
To conclude, we now need to use that $\E^k$ converges to an $N$-set $\E$. Since $\set*{\nupla u^k_i>t}\subset\nupla E^k_i$ whenever $t>0$ for each $k\in\N$ and $i=1,\dots,N$ by construction, and since $\nupla E^k_i\to\nupla E_i$ as $k\to\infty$ in $L^1(\Omega)$ for $i=1,\dots,N$, we easily get that $\set*{\nupla u_i>t}\subset\nupla E$ for a.e.\ $t>0$. 
Consequently, we must have that $\nupla F_i=\set*{\nupla u_i>0}\subset\nupla E_i$ for each $i=1,\dots,N$.
Thus $\lambda_{1,p}(\nupla E)\le\lambda_{1,p}(\nupla F)$, which, paired with~\eqref{eq:cerbiatto} and owing to \ref{p:m}, yields
\begin{equation*}
\Phi(\lambda_{1,p}(\nupla E))
\le 
\liminf_{k\to\infty}
\lstort^{\Phi_k,N}_{1,p}(\Omega).
\end{equation*}
Under the assumptions of \cref{res:stability}, the right hand side of the above inequality equals $\lstort^{\Phi, N}_{1,p}(\Omega)$.
Therefore $\nupla E$ is a $(p,\Phi)$-eigen-$N$-set of $\Omega$.
\end{proof}

\begin{remark}
If, on top of asking that the limit reference function $\Phi$ satisfies \ref{p:m}, one assumes that the whole sequence of reference functions possesses this property, then \cref{res:conv1_stability} (in virtue of the validity of the second part of \cref{res:H=lstort}) provides a stronger version of \cref{res:conv3_stability} in the case $p=1$. 
\end{remark}

\begin{remark}
The results of Section~\ref{sec:stability} may be achieved in more general contexts following~\cite{FPSS24}, as soon as suitable notions of isoperimetric inequality and Sobolev spaces are available.
We refer the reader to~\cite{FPSS24}*[Sects.~2.3.3 and~6.1], and also to~\cite{FPSS24}*[Sect.~7.1] for a plainer approach in the setting of metric-measure spaces. 
\end{remark}






\end{document}